\documentclass[11pt]{article}

\usepackage{amsmath,amsthm}

\usepackage{amssymb,latexsym}

\usepackage{enumerate}

\usepackage[mathscr]{euscript}

\newcommand{\BB}{{\mathcal  B}}
\newcommand{\DD}{{\mathcal  D}}
\newcommand{\EE}{{\mathcal  E}}
\newcommand{\FF}{{\mathcal  F}}

\newcommand{\II}{{\mathcal  I}}

\newcommand{\LL}{{\mathcal  L}}
\newcommand{\TT}{{\mathcal  T}}
\newcommand{\MM}{{\mathcal  M}}

\newcommand{\WW}{{\mathcal  W}}

\newcommand{\BC}{{\mathbb C}}

\newcommand{\BF}{{\mathbb F}}

\newcommand{\BN}{{\mathbb N}}
\newcommand{\BQ}{{\mathbb Q}}
\newcommand{\BR}{{\mathbb R}}
\newcommand{\BX}{{\mathbb X}}

\newcommand{\fch}{{\mathbf{1}}}
\newcommand{\arrowd}{\mathop{\rightarrow}_d}
\newcommand{\arrowp}{\mathop{\rightarrow}_P}
\newcommand{\arrowdl}{\mathop{\longrightarrow}_d}
\newcommand{\arrowpl}{\mathop{\longrightarrow}_P}
\newcommand{\arroww}{\mathop{\rightarrow}_{w}}
\newcommand{\arrowwl}{\mathop{\longrightarrow}_{w}}

\newtheorem{theorem}{\bf Theorem}[section]
\newtheorem{proposition}[theorem]{\bf Proposition}
\newtheorem{corollary}[theorem]{\bf Corollary}

\theoremstyle{definition}
\newtheorem{definition}[theorem]{Definition}
\newtheorem{example}[theorem]{\bf Example}

\newtheorem{remark}[theorem]{Remark}

\numberwithin{equation}{section}

\begin{document}

\title{Stability of solutions of semilinear evolution equations
with integro-differential operators}
\author {Andrzej Rozkosz and Leszek S\l omi\'nski\\
	{\small Faculty of Mathematics and Computer Science,
		Nicolaus Copernicus University} \\
	{\small  Chopina 12/18, 87--100 Toru\'n, Poland}\\
	{\small E-mail addresses: rozkosz@mat.umk.pl (A. Rozkosz), leszeks@mat.umk.pl (L. S\l omi\'nski)}}
\date{}
\maketitle

\begin{abstract}
We consider solutions of the Cauchy problem for semilinear equations with (possibly) different
L\'evy operators. We provide various results on their  convergence under the assumption that   symbols of the involved operators converge to the symbol of some L\'evy operator. Some results are proved for a more general class of pseudodifferential operators.
\end{abstract}

\footnotetext{{\em Mathematics Subject Classification:}
Primary 35K58, 35R11; Secondary 60H30.}

\footnotetext{{\em Keywords:} L\'evy-type operator, Cauchy problem, semilinear equation, convergence of solutions.
}

\section{Introduction}
\label{sec1}

Let $\varphi:\BR^d\rightarrow\BR$, $f:[0,T]\times\BR^d\times\BR$ be measurable functions  and $u$ be a solution of the semilinear Cauchy problem
\begin{equation}
\label{eq1.1}
\begin{cases}
\partial_su+L u=-f(s,x,u),\quad (s,x)\in Q_T:=(0,T)\times\BR^d,\\
u(T,x)=\varphi(x),\quad x\in\BR^d,
\end{cases}
\end{equation}
where $L$ is a L\'evy-type integro-differential operator of the form
\begin{equation}
\label{eq1.4}
L=\LL+\II.
\end{equation}
Here, for $u\in C^{\infty}_c(\BR^d)$,
\[
\LL u(x)=\sum^d_{i=1}b_i(x)\partial_{x_i}u(x)
+\frac{1}{2}\sum^d_{i,j=1}a_{ij}(x)\partial^2_{x_ix_j}u(x),
\qquad a=\sigma\sigma^T,
\]
for some  bounded Lipschitz continuous coefficients $\sigma:\BR^d\rightarrow\BR^d\times\BR^d$, $b:\BR^d\rightarrow\BR^d$, and
\[
\II u(x)=\int_{\BR^d}(u(x+y)-u(x)-y\cdot\nabla u(x)
\fch_{|y|\le1})\,\nu(dy)
\]
for some L\'evy measure $\nu$ on $\BR^d$. Under these assumptions on $\sigma,b,\nu$  the operator $L$ generates a conservative Feller semigroup such that $C^{\infty}_c(\BR^d)\subset D(L)$.

Suppose that $\{L^n=\LL^n+\II^n\}$ is  a sequence of  operators of the same form as $L$ such that each $L^n$ generates a Feller semigroup  and for each $n$ there exists a unique solution of (\ref{eq1.1}) with $L$ replaced by $L^n$.
The aim of the present paper is to show that
if $L^n\rightarrow L$ in the sense described by (\ref{eq1.2}) below, then
\begin{equation}
\label{eq1.3}
u_n\rightarrow u,
\end{equation}
and the form of convergence  (pointwise, almost everywhere, weak in $L^2$, etc.) depends on the assumptions on the data $\varphi$ and $f$. We are mainly  interested in the  case where the nonlocal parts $\II^n$ are present. The model  simple example is
\begin{equation}
\label{eq1.5}
L=\Delta,\qquad L^n=\II^n=\Delta^{\alpha_n/2}\quad\mbox{with }\alpha_n\rightarrow 2^{-}
\end{equation}
(the fractional Laplace operator  $\Delta^{\alpha_n/2}:=-(-\Delta)^{\alpha_n/2}$ corresponds to the L\'evy measure $\nu_n(dy)=c_{d,\alpha_n}|y|^{-d-\alpha_n}\,dy$ for some constant $c_{d,\alpha_n}>0$). This model example  was recently addressed in \cite{XCV}. In that paper it is proved that under some assumptions on $\varphi,f$ the convergence (\ref{eq1.3}) holds  weakly star in $L^{\infty}(0,T;L^2(\BR^d))$ and weakly in $L^2(0,T;L^2(\BR^d))$.

In the present paper,  we consider much more general situation and give stronger convergence results even in the model example (\ref{eq1.5}) considered in \cite{XCV}. To describe them it is convenient to regard $L,L^n$ as  pseudodifferential operators. Since $C^{\infty}_c(\BR^d)\subset D(L)$, by Courr\`ege theorem, $L$ restricted
to $C^{\infty}_c(\BR^d)$ has the form
\[
Lu(x)=-p(x,D)u(x):=-\int_{\BR^d}e^{i(x,\xi)}p(x,\xi)\hat u(\xi)\,d\xi,
\quad u\in C^{\infty}_c(\BR^d),
\]
with  symbol $p(x,\xi):\BR^d\times\BR^d\rightarrow\BC$ defined by
\begin{equation}
\label{eq1.6}
p(x,\xi)=-i(b(x),\xi)+\frac12(a(x)\xi,\xi)+\psi(\xi),\quad x,\xi\in\BR^d,
\end{equation}
where $\psi:\BR^d\rightarrow\BC$ is the characteristic exponent of the pure-jump L\'evy process with L\'evy measure $\nu$, i.e. $\psi$ is given by the L\'evy--Khinchine formula
\[
\psi(\xi)=\int_{y\neq0}(1-e^{i(\xi,y)}+i(\xi,y)\fch_{\{|y|\le1\}})\,\nu(dy).
\]
In general, by a solution of (\ref{eq1.1}) we mean a probabilistic solution, i.e. a function $u:[0,T)\times\BR^d$ which satisfies the following ``nonlinear Feynman--Kac formula":
\begin{equation}
\label{eq1.7}
u(s,x)=E\Big(\varphi(X^{s,x}_T)
+\int^T_sf(t,X^{s,x}_t,u(r,X^{s,x}_t))\,dt\Big),\quad (s,x)\in[0,T)\times\BR^d,
\end{equation}
where $X^{s,x}$ is the jump-diffusion process with generator $L$ starting at time $s$ from $x$.

Let $p_n$ be the symbol of $L^n$, which is defined by (\ref{eq1.6}) with $a,b,\psi$ replaced by some $a^n,b^n$ and $\psi_n$. We assume that $a^n,b^n$ are uniformly bounded and Lipschitz continuous, and
\begin{equation}
\label{eq1.2}
a^n\rightarrow a,\,b^n\rightarrow b\mbox{ locally uniformly on }\BR^d,\quad \psi_n\rightarrow\psi\mbox{ pointwise in }\BR^d.
\end{equation}
As for $f$, we assume that $(x,y)\rightarrow f(t,x,y)$ is continuous, $y\rightarrow f(t,x,y)$ is nonincreasing (in fact this condition can be somewhat relaxed),  $|f(t,x,y)-f(t,x,0)\le h(t,x) H(|y|)$ for some nonnegative $h$ and nonnegative nondecreasing $H$, $f(t,x,0)$ is bounded  or satisfies some integrability conditions. Examples are
\begin{equation}
\label{eq1.8}
f(t,x,y)=h_1(t,x)+h_2(t,x)g(y),\quad g(y)=-y|y|^{p-1}\quad\mbox{or}\quad g(y)=1-e^y,
\end{equation}
where 
$h_1,h_2$ are continuous in $x$, $h_2$ is bounded and nonnegative and $h_1$ is bounded or satisfies some integrability condition.
We show that then, if $\varphi$ is bounded continuous and  $f(\cdot,\cdot,0)$ is bounded, then the probabilistic solutions $u_n$ of (\ref{eq1.1}) with $L$ replaced by $L^n$ (they exist and are unique) converge pointwise to the probabilistic solution  $u$.

The assumptions on $\varphi,f(\cdot,\cdot,0)$ can be relaxed if $a^n,b^n$ are constant, i.e. if $L^n$ are L\'evy operators, and, roughly speaking, $\psi_n$ satisfy the Hartman--Wintner condition uniformly in $n$. One can then show that some form of convergence (\ref{eq1.3}) holds for $\varphi,  f(\cdot,\cdot,0)$ satisfying  integrability conditions. Furthermore, for  such subclass of operators
$u_n,u$ become weak solutions (for $L^2$-data) or renormalized solutions (for $L^1$-data) and we get results on stability of weak or renormalized solutions.

In the present paper, we also consider operators with symbols of the form  (\ref{eq1.6}) with $\psi(\xi)$ replaced by $\psi(\gamma^T\xi)$ for some bounded Lipschitz $\gamma:\BR^d\rightarrow\BR^{d\times k}$. This means that we allow  operators (\ref{eq1.4}) with the integral part defined as $\II$, but with the measure $\nu$ depending  in a suitable way on $x$. As an example may serve the operator
\[
Lu(x)=|\gamma(x)|^{\alpha}\Delta^{\alpha/2}u(x),\quad u\in C^{\infty}_c(\BR^d).  
\]
For such operators we prove some convergence results but only for bounded continuous $\varphi$ and bounded $f(\cdot,\cdot,0)$.

To prove our results we use probabilistic methods. The solution $u_n$ associated with $L^n$ is given by (\ref{eq1.7}) with $X^{s,x}$ replaced by the process $X^{s,x,n}$ with generator $L^n$. Therefore, roughly speaking, the proof of (\ref{eq1.3}) reduces to the proof that the right-hand side of formula (\ref{eq1.7}) representing $u_n$ (it involves $u_n$ and $X^{s,x,n}$) converges in a suitable way to the right-hand side of (\ref{eq1.7}).
To this end, we show the convergence of solutions of backward stochastic differential equations (BSDEs) associated with the Cauchy problem involving $L^n$ to the solution of BSDE associated with (\ref{eq1.1}).
In fact, we show some abstract, more general than needed stability result for solution of BSDEs, which may be of independent interest.
It is worth noting that our stability result pertains to the BSDEs with different filtrations (generated by the processes $X^{s,x,n}$ in our applications) and its proof uses in an essential way the notion of weak convergence of filtrations introduced in \cite{CMS}.

\section{Stability of solutions of  BSDEs}
\label{sec2}

In  proofs of our main theorems  we  will need a stability result for solutions of
Markov-type backward stochastic differential equations (BSDEs in abbreviation)  with forward driving processes associated with operators $L^n$.  However, it is convenient to state it for more general abstract  BSDEs.

Let $(\Omega,\FF,\BF=(\FF_t)_{t\ge0},P)$ be a filtered probability space satisfying the usual conditions.
We denote by
$\DD$  the space of all $\BF$-progressively measurable c\`adl\`ag processes.
$\MM$ (resp. $\MM_{loc}$) is the space of all c\`adl\`ag
$(\BF,P)$-martingales (resp. local martingales) $M$ such that
$M_0=0$.

Recall that a  c\`adl\`ag $\BF$-adapted process $Y$ is of
Doob class (D)  if the collection $\{Y_{\tau},\tau\in\TT\}$,
where $\TT$ is the set a finite valued $\BF$-stopping times,
is uniformly integrable.

In what follows $T>0$ is an arbitrary but fixed number, $\xi$ is an $\FF_T$-measurable random
variable, $X$ is a process of class $\DD$ and $f:[0,T]\times\BR^d\times\BR\rightarrow\BR$ is a
measurable function.
As for $\xi$ and $f$ we will need the following hypotheses.

\begin{enumerate}
\item[(H1)]$\BR^d\times\BR\ni(x,y)\mapsto f(t,x,y)$ is continuous for every $t\in[0,T]$.

\item[(H2)]There is $\mu\ge0$ such that $(f(t,x,y)-f(t,x,y'))(y-y')\le\mu|y-y'|^2$ for all $t\in[0,T]$, $x\in\BR^d$ and $y,y'\in\BR$.

\item[(H3)]$E\int^T_0F_r(t,X_t)\,dt<\infty$ for every $r>0$,  where   $F_r(t,x)=\sup_{|y|\le r}|f(t,x,y)-f(t,x,0)|$.

\item[(H4)]$E|\xi|+E\int^T_0|f(t,X_t,0)|\,dt<\infty$.
\end{enumerate}

\begin{definition}
A pair $(Y,M)$ of processes on $[0,T]$ is a solution of $\mbox{BSDE}(\xi,X,f,\BF)$ if $Y\in\DD$, $Y$ is of class (D), $M\in\MM_{loc}$ and
\begin{equation}
\label{eq2.17}
Y_t=\xi+\int^T_tf(s,X_s,Y_s)\,ds-\int^T_tdM_s,\quad t\in[0,T],\,P\mbox{-a.s.}
\end{equation}
\end{definition}

\begin{theorem}
\label{th2.2}
If \mbox{\rm(H1)--(H4)} are satisfied, then there exists  a unique solution $(Y,M)$ of $\mbox{BSDE}(\xi,X,f,\BF)$ such that
$M$ is a uniformly integrable martingale.
\end{theorem}
\begin{proof}
See \cite[Theorem 3.6]{KR:CM} for the existence part and \cite[Corollary 3.2]{KR:JFA} for uniqueness. 
\end{proof}

Suppose now that for each $n\ge1$  we are given a filtration $\BF^n=(\FF^n_t)_{t\ge0}$  such that $\FF^n_t\subset\FF$ for $t\ge0$, an $\FF^n_T$-measurable random variable $\xi^n$ and a process $X^n$ of class $\DD$ with respect to $\BF^n$. Suppose also that for each $n\ge1$ there exists a unique solution $(Y^n,M^n)$  of $\mbox{BSDE}(\xi^n,X^n,f,\BF^n)$, i.e.
\begin{equation}
\label{eq2.18}
Y^n_t=\xi^n+\int^T_tf(s,X^n_s,Y^n_s)\,ds-\int^T_tdM^n_s,\quad t\in[0,T].
\end{equation}
In our stability result a key role will be played by the concept of weak convergence of filtration introduced in \cite{CMS}. To state it we denote by
$D([0,T];\BR^d)$ the space of $\BR^d$-valued functions on $[0,T]$ which are right continuous, with left limits. We equip it with the Skorokhod $J_1$-topology.
We use ``$\arrowp$" to denote the convergence in probability $P$.

The following definition is taken from \cite{CMS}. It is a version of the definition introduced in \cite{H}.

\begin{definition}
We say that $\{\BF^n\}_{n\ge1}$ converges weakly to $\BF$ on $[0,T]$ (and we write $\BF^n\arroww\BF$) if for every $B\in\FF_T$,
\[
E(\fch_B|\FF^n_{\cdot})
\arrowpl E(\fch_B|\FF_{\cdot})\quad\mbox{in }D([0,T];\BR).
\]
\end{definition}

In the stability theorem our basic assumption on the processes associated with  operators $L^n$ is the following:
\begin{enumerate}
\item[(C)] $X^n\arrowp X$ in $[D([0,T];\BR)]^d$ and $\BF^n\arroww\BF$ on $[0,T]$.
\end{enumerate}

\begin{remark}
\label{rem2.4}
(i) By \cite[Remark 1]{CMS}, $\BF^n\arroww\BF$ on $[0,T]$ if and only if $E(X|\FF^n_{\cdot})\rightarrow E(X|\FF_{\cdot})$ in $D([0,T];\BR)$ for every $\FF_T$-measurable integrable random variable.
\smallskip\\
(ii) From part (i) it follows that if $X$ is an $\FF_T$-measurable integrable random variable, $X_n\rightarrow X$ in $L^1(P)$ and $\BF^n \arroww\BF$ on $[0,T]$, then $E(X_n|\FF^n_{\cdot})\arrowp E(X|\FF_{\cdot})$ in $D([0,T];\BR)$.
\\
(iii) In \cite{CMS} it is shown that for some interesting classes of processes with if $X^n\arrowp X$ in $[D([0,T];\BR)]^d$, then $\BF^n \arroww\BF$ on $[0,T]$. For instance, by \cite[Proposition 2]{CMS}, this holds true if the processes $X^n$ have independent increments.
\end{remark}

We will also need the following assumptions relating $\{X^n\}$ to $f$ and $\varphi$ (below $dt$ denotes the Lebesgue measure on $(0,T)$):

\begin{enumerate}
\item[(H5)]For every $r>0$ the sequence $\{F_r(\cdot,X^n_{\cdot})\}_{n\ge1}$ is uniformly $dt\otimes P$-integrable on $(0,T)\times\Omega$.

\item[(H6)]$\{f(\cdot,X^n,0)\}_{n\ge1}$ is uniformly $dt\otimes P$-integrable on $(0,T)\times\Omega$.

\item[(H7)]$\xi^n\longrightarrow\xi$ in $L^1(P)$, where $\xi^n=\varphi(X^n)$, $\xi=\varphi(X)$.
\end{enumerate}

It is well known 
that the study of backward equations with $f$ satisfying (H2) can be reduced by a simple transformation of variables to the study of equations satisfying (H2) with $\mu=0$.  For the reader's convenience, we describe the transformation below.

\begin{remark}
\label{rem2.5}
Let $\tilde\xi=e^{\mu T}\xi$,
 $\tilde f(t,x,y)=e^{\mu t}f(t,x,e^{-\mu t}y)-\mu y$.\smallskip\\
(i) If $\xi,f$ satisfy (H1), (H4), then clearly $\tilde\xi,\tilde f$ satisfy these hypotheses. If $f$ satisfies (H2), then $\tilde f$ satisfies (H2) with $\mu=0$.
Define $\tilde F_r$ as $F_r$, but with $f$ replaced by $\tilde f$. Then
$\tilde F(t,x)\le e^{\mu t}F_{r'}(t,x)+\mu r$ with $r'=e^{-\mu t}r$, so if $f$ satisfies (H3). then $\tilde f$ satisfies (H3).
\smallskip\\
(ii) An application of It\^o's formula shows that  $(Y,M)$ is a solution of (\ref{eq2.17}) if and only if  $(\tilde Y_t,\tilde M_t)):=(e^{\mu t}Y_t, e^{\mu t}M_t)$, $t\in[0,T]$, is a solution of  $\mbox{BSDE}(\tilde\xi,X,\tilde f,\BF)$.
\smallskip\\
(iii) If $f$ satisfies (H5) (resp. (H6)), then $\tilde f$ satisfies (H5) (resp. (H6)).
\end{remark}

\begin{theorem}
\label{th2.4}
Assume that \mbox{\rm(C)} and \mbox{\rm(H1), (H2), (H5)--(H7)} are satisfied. Then
\begin{equation}
\label{eq2.1}
(Y^n,M^n)\arrowpl (Y,M)\quad\mbox{in }D([0,T];\BR^2).
\end{equation}
\end{theorem}
\begin{proof}
By Remark \ref{rem2.5}, without loss of generality we may and will assume that $\mu=0$. Note also that (H5), (H6) together with (H1), (C) imply (H3), (H4).\\
{\em Step 1}. We assume additionally that
\begin{equation}
\label{eq2.2}
\sup_{n\ge1}\Big(|\xi_n|+\sup_{0\le t\le T}|f(t,X^n_t,0)|
+\sup_{0\le t\le T}|f(t,X_t,0)|\Big)\le C
\end{equation}
for some $C\ge0$ and there is $L\ge0$ such that
\begin{equation}
\label{eq2.3}
|f(t,x,y)-f(t,x,y')|\le L|y-y'|,\quad t\in[0,T],\, x\in\BR^d,\,y,y'\in\BR.
\end{equation}
The proof of the theorem under these assumptions is a modification of the proof of \cite[Theorem 4]{CMS} in the case where, in the notation adopted in \cite{CMS},
$A^n_s=A_s=s$, $X^n=\xi_n,X=\xi$ and $ g^n_s(\cdot)=f(s,X^n_s,\cdot),
g_s(\cdot)=f(s,X_s,\cdot)$.
As compared with the proof of \cite[Theorem 4]{CMS}, in the present case we do not know that $g^n\rightarrow g$ uniformly in $s$ and $x$.
For equations (\ref{eq2.17}) and (\ref{eq2.18}) we consider  Picard approximations of the form
\[
U^0_t=0,\qquad U^k_t=E\Big(\xi+\int^T_tf(s,X_s,U^{k-1}_s)\,ds\,\big|\,\FF_t\Big),
\]
\[
U^{n,0}_t=0,\qquad U^{n,k}_t=E\Big(\xi_n+\int^T_tf(s,X^n_s,U^{n,k-1}_s)\,ds\,\big|\,\FF^n_t\Big),
\quad t\in[0,T],\,k\ge1.
\]
By \cite[Theorem 2.4]{A} and (\ref{eq2.2}), (\ref{eq2.3}) we have
\[
E\int^T_0|U^{n,k+1}_s-U^{n,k}_s|\,ds\le \frac{(LT)^{k+1}}{(k+1)!}E\int^T_0|U^{n,1}_s|\,ds
\le\frac{(LT)^{k+1}}{(k+1)!}\cdot CT,
\]
which implies that
\[
E\int^T_0|Y^n_s-U^{n,k}_s|\,ds\le \sum^{\infty}_{p=k+1}\frac{(LT)^{k+1}}{(k+1)!}\cdot CT,
\quad n,k\ge1.
\]
As in the proof of \cite[(9)]{CMS}, we deduce from the above inequality that  for every $\varepsilon>0$,
\begin{equation}
\label{eq2.19}
\lim_{k\rightarrow\infty}\sup_{n\ge1}
P(\sup_{t\le T}|Y^n_t-U^{n,k}_t|>\varepsilon)=0.
\end{equation}
In much the same way we obtain
\begin{equation}
\label{eq2.20}
\sup_{t\le T}|Y_t-U^k_t|\arrowp0.
\end{equation}
We next show that be the  weak convergence of filtrations we have
\begin{equation}
\label{eq2.21}
U^{n,k}\arrowpl U^k\quad\mbox{in }D([0,T];\BR),\quad k\ge1.
\end{equation}
To see this, we first observe that by (C) and (\ref{eq2.2}),
\[
E\Big|\int^T_0f(t,X^n_t,0)\,dt-\int^T_0f(t,X_t,0)\,dt\Big|\rightarrow0.
\]
Therefore, by Remark \ref{rem2.4}(ii),
\[
E\Big(\xi_n+\int^T_0f(s,X^n_s,0)\,ds\,\big|\,\FF^n_{\cdot}\Big)
\arrowpl E\Big(\xi+\int^T_0f(s,X_s,0)\,ds\,\big|\,\FF_{\cdot}\Big)
\quad\mbox{in }D([0,T];\BR),
\]
which means that $U^{n,1}\arrowp U^1$ in $D([0,T];\BR)$. We next show by induction that for every $k\ge1$ the convergence $U^{n,k-1}\arrowp U^{k-1}$ implies (\ref{eq2.21}), which together with (\ref{eq2.19}) and (\ref{eq2.20})  shows that
\begin{equation}
\label{eq2.11}
Y^n\arrowpl Y\quad\mbox{in }D([0,T];\BR^d).
\end{equation}
To show the joint convergence (\ref{eq2.1}) we first note that by (\ref{eq2.11}),  $\{\sup_{t\leq T}|Y^n_t|\}$ is bounded in probability.	Moreover, by using  (\ref{eq2.2})  and (\ref{eq2.3}) we get
\[
h^n:=\sup_{s\leq T}|f(s,X^n_s,Y^{n}_s)-f(s,X_s,Y_s)|\leq 2C+L(\sup_{t\leq T}|Y^n_t|+\sup_{t\leq T}|Y_t|)
\]
for $n\ge1$, which implies that  the sequence $\{h^n\}$  is also bounded in probability. By this, (H5)  and (\ref{eq2.11}),
\begin{align}\nonumber
\sup_{t\le T} &\Big|\int_0^tf(s,X^n_s,Y^n_s)\,ds
-\int^t_0f(s,X_s,Y_s)\,ds\Big|\\
&\qquad\leq\int_0^T|f(s,X^n_s,Y^n_s)-f(s,X_s,Y_s)|\,ds\arrowpl0.
\label{eq2.212}
\end{align}
Since
	\[
	M_t=Y_t-Y_0+\int^t_0f(s,X_s,Y_s)\,ds,\quad M^n_t=Y^n_t-Y^n_0+\int^t_0f(s,X^n_s,Y^n_s)\,ds,
	\]
it is clear that the joint convergence (\ref{eq2.1}) follows from (\ref{eq2.11}) and (\ref{eq2.212}).
\\
{\em Step 2}. We shall show how to dispense with the assumption (\ref{eq2.3}) (but we maintain (\ref{eq2.2})). For $m\ge1$ we set
	\[
	f_m(t,x,y)=\inf_{z\in\BQ}\{m|y-z|+f(t,x,z)\}
	\]
($\BQ$ is the set of rational numbers). The approximations $f_m$, $m\geq1$, have the following properties (see, e.g., \cite[Lemma 1]{LSM}):
	\begin{equation}
		\label{eq2.4}
		\quad |f_m(t,x,y)-f_m(t,x,y')|\le m|y-y'|,\quad   t\in[0,T], x\in\BR^d, y,y'\in\BR,
	\end{equation}
	\begin{equation}
		\label{eq2.5}
		f_1(t,x,y)\le f_m(t,x,y)\le f(t,x,y), \quad t\in[0,T], x\in\BR^d, y\in\BR,
	\end{equation}
Furthermore,  for all $t\in[0,T]$ and $x\in\BR^d, y\in\BR$,
	\begin{align*}
		f_m(t,x,y)\nearrow f(t,x,y)  \mbox{ and if } (x_m,y_m)\to(x,y)\,\,\mbox{\rm then }
		f_m(t,x_m,y_m)\rightarrow f(t,x,y).
\end{align*}
Note also that by (\ref{eq2.2})  and (\ref{eq2.5}), for all $t\in[0,T]$, $x\in\BR^d$, $y\in\BR$, $m,n\geq1$ we have
	\begin{equation}\label{eq2.52}
		|f_m(t,X^n_t,0)|\le C \quad\mbox{\rm and}\quad\sup_{|y|\le r}|f_m(t,X^n_t,y)|\le r+C+\sup_{|y|\le r}|f(t,X^n_t,y)|.
	\end{equation}
	By Theorem \ref{th2.2}, for any $n,m\ge1$  there exists a unique solution $(Y^{n,(m)},M^{n,(m)})$ of $\mbox{BSDE}(\xi_n,X^n,f_m,\BF^n)$. Note that by (\ref{eq2.2}) and (\ref{eq2.52}),
	\begin{equation}
		\label{eq2.6}
		|Y^{n,(m)}|,\,|Y^n|\le C
	\end{equation}
	(see \cite[(2.10)]{KR:JFA}). Furthermore, for every $m\ge1$ there also exists a unique solution $(Y^{(m)},M^{(m)})$ of $\mbox{BSDE}(\xi,X,f_m,\BF)$ such that
	\begin{equation}
		\label{eq2.7}
		Y^{(m)}_t\nearrow Y_t\quad P\mbox{-a.s.},\quad \sup_{t\le T}|Y^{(m)}_t-Y_t|\arrowpl0,
		\quad \sup_{t\le T}|M^{(m)}_t-M_t|\arrowpl0
	\end{equation}
	(see \cite[Lemma 2.6]{KR:JFA} and its proof).  By {\em Step 1}, for every $m\ge1$,
	\begin{equation}
		\label{eq2.8}
		(Y^{n,(m)},M^{n,(m)})\arrowpl(Y^{(m)},M^{(m)})\quad\mbox{in }D([0,T];\BR^2).
	\end{equation}
	By It\^o's formula for convex functions (see, e.g., \cite[page 216]{P}),
	\begin{align*}
		|Y^n_t-Y^{n,(m)}_t|&\le\int^T_t\mbox{sgn}(Y^n_s-Y^{(n,(m)}_s)
		(f(s,X^n_s,Y^n_s)-f_m(s,X^n_s,Y^{n,(m)}_s))\,ds\\
		&\quad-\int^T_t\mbox{sgn}(Y^n_{s-}-Y^{n,(m)}_{s-})\,d(M^n_s-M^{n,(m)}_s),
	\end{align*}
where $\mbox{sgn}(x)=1 $ if $x>0$ and $\mbox{sgn}(x)=-1 $ if $x\le0$.	Since $Y^n_t\ge Y^{n,(m)}_t$, assumption (H2) implies that
	\[
	|Y^n_t-Y^{n,(m)}_t|\le\int^T_t
	(f(s,X^n_s,Y^n_s)-f_m(s,X^n_s,Y^{n,(m)}_s))\,ds-\int^T_td(M^n_s-M^{n,(m)}_s).
	\]
	Hence, for every $t\in[0,T]$,
	\begin{equation}
	\label{eq2.22}
	|Y^n_t-Y^{n,(m)}_t|\le E\Big(\int^T_0
	|f(s,X^n_s,Y^n_s)-f_m(s,X^n_s,Y^{n,(m)}_s)|\,ds\,\big|\,\FF^n_t\Big).
	\end{equation}
	By this estimate and \cite[Lemma 6.1]{BDHPS}, for $\beta\in(0,1)$ we have
	\begin{equation}
		\label{eq2.9}
		E\sup_{t\le T}|Y^n_t-Y^{n,(m)}_t|^{\beta}\le(1-\beta)^{-1}
		\Big(E\int^T_0
		|f(s,X^n_s,Y^n_s)-f_m(s,X^n_s,Y^{n,(m)}_s)|\,ds\Big)^{\beta}.
	\end{equation}
	Observe that by  (\ref{eq2.2}), (\ref{eq2.4}), (\ref{eq2.52}) and (\ref{eq2.6}),
	\begin{align*}
		&|f(s,X^n_s,Y^{n,(m)}_s)-f_m(s,X^n_s,Y^{n,(m)}_s)|
		=f(s,X^n_s,Y^{n,(m)}_s)-f_m(s,X^n_s,Y^{n,(m)}_s)\\
		&\qquad \le f(s,X^n_s,Y^{n,(m)}_s)-f_1(s,X^n_s,Y^{n,(m)}_s)\\
		&\qquad\le F^n_C(s,X^n_s)+|f(s,X^n_s,0)|+|f_1(s,X^n_s,Y^{n,(m)}_s)|\\
		&\qquad\le F^n_C(s, X^n_s)+|f(s,X^n_s,0)|+|Y^{n,(m)}_s|+|f_1(s,X^n_s,0)|
		\le F^n_C(s,X^n_s)+3C.
	\end{align*}
	By this and (H5), we can pass to the limit on the right-hand side of (\ref{eq2.9}). We then get
	\begin{align}
		\label{eq2.10}	
		\lim_{m\rightarrow\infty}&\limsup_{n\rightarrow\infty}
		E\sup_{t\le T}|Y^n_t-Y^{n,(m)}_t|^{\beta} \nonumber\\
		&\leq\lim_{m\to\infty}(1-\beta)^{-1}
		\Big(E\int^T_0
		|f(s,X_s,Y_s)-f_m(s,X_s,Y^{(m)}_s)|\,ds\Big)^{\beta}=0.
	\end{align}
	Using (\ref{eq2.22}), one can also show that
	$\lim_{m\rightarrow\infty}\limsup_{n\rightarrow\infty}
	E|Y^n_t-Y^{n,(m)}_t|=0$ for $t\in[0,T]$.
	 Since
	$M^n_t=Y^n_t-Y^n_0+\int^t_0f(s,X^n_s,Y^n_s)\,ds$ and similarly $M^{n,(m)}_t=Y^{n,(m)}_t-Y^{n,(m)}_0+\int^t_0f_m(s,X^n_s,Y^{n,(m)}_s)\,ds$,
	it is  clear that we  have
	\[
	\lim_{m\rightarrow\infty}\limsup_{n\rightarrow\infty}
	E|M^n_T-M^{n,(m)}_T|=0.
	\]
	Using once again  \cite[Lemma 6.1]{BDHPS}, for $\beta\in(0,1)$ we get
	\begin{equation}
	\label{eq2.16}
	\lim_{m\rightarrow\infty}\limsup_{n\rightarrow\infty}
	E\sup_{t\le T}|M^n_t-M^{n,(m)}_t|^{\beta}=0.
	\end{equation}
	Putting together (\ref{eq2.7}), (\ref{eq2.8}) and (\ref{eq2.10}), (\ref{eq2.16}) we obtain (\ref{eq2.1}).\\
	{\em Step 3}. The general case. For $k\ge1$ set $T_k(x)=(-k)\vee(x\wedge k)$ and then
	\[
	\xi^{(k)}=T_k(\xi),\quad \xi^{n,(k)}=T_k(\xi_n),
	\]
	\begin{equation}
		\label{eq2.14}
		f^k(t,x,y)=f(t,x,y)-f(t,x,0)+T_k(f(t,x,0)).
	\end{equation}
	Let $(Y^{n,(k)},M^{n,(k)})$ be the solution of $\mbox{BSDE}(\xi^{(n,(k)},X^n,f^k,\BF)$, i.e.
	\[
	Y^{n,(k)}_t=\xi^{n,(k)}+\int^T_tf^k(s,X^n_s,Y^{n,(k)}_s)\,ds-\int^T_tdM^{n,(k)}_s, \quad t\in[0,T],
	\]
	and $(Y^{(k)},M^{(k)})$ be the solution of $\mbox{BSDE}(\xi^{(k)},X,f^k,\BF)$, i.e.
	\[
	Y^{(k)}_t=\xi^{(k)}+\int^T_tf^k(s,X_s,Y^{(k)}_s)\,ds-\int^T_tdM^{(k)}_s, \quad t\in[0,T],
	\]	
	From the proof of \cite[Theorem 2.7]{KR:JFA} it  follows that
	\begin{equation}
		\label{eq2.12}
		\sup_{t\le T}|Y^{(k)}_t-Y_t|\arrowpl0,\qquad \sup_{t\le T} |M^{(k)}_t-M_t|\arrowpl0.
	\end{equation}
On the other hand, from {\em Step 2} we know that for every $k\ge1$,
\begin{equation}
\label{eq2.13}
(Y^{n,(k)},M^{n,(k)})\arrowpl (Y^{(k)},M^{(k)})\quad\mbox{in }D([0,T];\BR^2).
\end{equation}
Since $Y^n-Y=(Y^n-Y^{(n,(k)})+(Y^{n,(k)}-Y^{(k)})+(Y^{(k)}-Y)$ (and we have similar decomposition of $M^n-M$), to complete the proof it suffices to estimate the difference $Y^n-Y^{n,(k)}$ (and $M^n-M^{n,(k)}$) uniformly in $n$. To this end, we first use the Meyer--Tanaka formula and (H2) to get
	\begin{align*}
		|Y^n_t-Y^{n,(k)}_t|&\le|\xi^n-\xi^{n(k)}|\\
		&\quad+\int^T_t\mbox{sgn}(Y^n_s-Y^{n,(k)}_s)
		(f(s,X^n_s,Y^n_s)-f^k(s,X^n_s,Y^{n,(k)}_s))\,ds\\
		&\quad-\int^T_td(M^n_s-M^{n,(k)}_s)\\
		&\le|\xi^n-\xi^{n,(k)}|+\int^T_t
		|f(s,X^n_s,Y^n_s)-f^k(s,X^n_s,Y^{n,(k)}_s)|\,ds\\&\quad-\int^T_td(M^n_s-M^{n,(k)}_s)\\
		&\le|\xi^n-\xi^{n,(k)}|+\int^T_t
		|f(s,X^n_s,0)|\fch_{\{|f(s,X^n_s,0)|>k\}}\,ds
		-\int^T_t\!d(M^n_s-M^{n,(k)}_s).
	\end{align*}
	Hence
	\begin{equation}
		\label{eq2.15}
		|Y^n_t-Y^{n,(k)}_t|\le E\Big(|\xi^n|
		\fch_{\{|\xi^n|>k\}}+\int^T_0|f(s,X^n_s,0)|
		\fch_{\{|f(s,X^n_s,0)|>k\}}\,ds\,\big|\,\FF^n_t\Big),
	\end{equation}
	so by (H6) and (H7), $\lim_{k\rightarrow\infty}\sup_{n\ge1}E|Y^n_t-Y^{n,(k)}_t|=0$, $t\in[0,T]$.  Furthermore, applying  \cite[Lemma 6.1]{BDHPS} we get
	\[
	E\sup_{t\le T}|Y^n_t-Y^{n,(k)}|^{\beta}\le(1-\beta)^{-1}
	\Big(E|\xi^n|\fch_{\{|\xi^n|>k\}}+E\int^T_0|f(s,X^n_s,0)|
	\fch_{\{|f(s,X^n_s,0)|>k\}}\,ds\Big)^{\beta}
	\]
	for $\beta\in(0,1)$. Using once again (H6) and (H7) we get
	\[
	\lim_{k\rightarrow\infty}\sup_{n\ge1}E\sup_{t\le T}|Y^n_t-Y^{n,(k)}_t|^{\beta}=0.
	\]
	Finally, similarly to  {\em Step 2},  by previously used arguments we deduce first that
	\[
	\lim_{k\rightarrow\infty}\sup_{n\ge1}E|M^n_T-M^{n,(k)}_T|=0
	\]
	and next, by \cite[Lemma 6.1]{BDHPS}, that  for $\beta\in(0,1)$ we have
	\[
	\lim_{k\rightarrow\infty}\sup_{n\ge1}E\sup_{t\le T}|M^n_t-M^{n,(k)}_t|^{\beta}=0,
	\]
	which together with (\ref{eq2.12}) and (\ref{eq2.13}) gives (\ref{eq2.1}) and completes the proof in the general case.
\end{proof}

\section{Probabilistic  solutions}

In what follows we assume that $\sigma:\BR^d\rightarrow\BR^{d\times d}$, $b:\BR^d\rightarrow\BR^d$, $\gamma:\BR^d\rightarrow\BR^{d\times k}$ are bounded Lipschitz continuous functions and  $\nu$ is a L\'evy measure on $\BR^d$, i.e.
a measure such that $\nu(\{0\})=0$ and $\int_{\BR^d}(1\wedge|y|^2)\,\nu(dy)<\infty$. We let $a=\sigma\cdot\sigma^T$, where $\sigma^T$ denotes the transpose of $\sigma$.

\subsection{SDEs and Feller semigroups}

Let $s\in[0,T)$, $x=(x_1,\dots,x_d)\in\BR^d$, $W=(W^1,\dots,W^d)$ be a standard Wiener process and $N=(N^1,\dots,N^k)$   be a pure-jump $k$-dimensional L\'evy proces independent of $W$ starting from 0. Consider the stochastic differential equation (SDE)
\begin{equation}
\label{eq3.1}
X^{s,x}_t=x+\int^t_sb(X^{s,x}_r)\,dr+\int^t_s\sigma(X^{s,x}_r)\,dW_r
+\int^t_s\gamma(X^{s,x}_{r-})\,dN_r,\quad t\in[s,T].
\end{equation}
It can be rewritten as
\[
X^{s,x,i}_t=x_i+
\sum^{d+1+k}_{j=1}\int^t_s\Phi_{ij}(X^{s,x}_{r-})\,dZ^{j}_r,
\quad t\in[s,T],\quad i=1,\dots,d,
\]
or succinctly as
\begin{equation}
\label{eq3.18}
X^{s,x}_t=x+\int^t_s\Phi(X^{s,x}_{r-})\,dZ_r, \quad t\in[s,T],
\end{equation}
where
\[
Z_t=(t, W_t,N_t),\,\,t\ge0,\qquad
\Phi_{ij}(x)=\begin{cases}
b_i(x) & j=1,\\
\sigma_{ij}(x), & j=2,\dots,d+1,\\
\gamma_{i,j-d-1}(x), & j=d+2,\dots,d+1+k.
\end{cases}
\]
It is well known (see, e.g., \cite[Theorem V.7]{P})  that for each $(s,x)\in[0,T)\times \BR^d$
there exists a unique strong solution $X^{s,x}$ of (\ref{eq3.1}). Furthermore,  by
\cite[Corollary 3.3]{SS} (see also  \cite[Theorem 2.49]{Sc}), $\BX=(X^{0,x})_{x\in\BR^d}$ is a Feller process, i.e.   it is a Markov process whose transition semigroup defined by
\[
P_tu(x)=Eu(X^{0,x}_t),\quad u\in\BB_b(\BR^d),
\]
is a Feller semigroup.   Let $(L, D(L))$ denote its  generator. By \cite[Theorem 3.5]{SS}
(or \cite[Theorem 2.50]{Sc}), $C^{\infty}_c(\BR^d)\subset D(L)$,
where $C^{\infty}_c(\BR^d)$ is  the space of smooth functions on $\BR^d$ with compact support.
Therefore, by the Courr\`ege theorem,  $L$  restricted to $C_c^{\infty}(\BR^d)$ is a pseudodifferential operator of the form
\begin{equation}
\label{eq3.3}
Lu(x)=-p(x,D)u(x):=-\int_{\BR^d}e^{i(x,\xi)}p(x,\xi)\hat u(\xi)\,d\xi,
\quad u\in C^{\infty}_c(\BR^d),
\end{equation}
with some symbol $p(x,\xi):\BR^d\times\BR^d\rightarrow\BC$ (in (\ref{eq3.3}), $\hat u$ is the Fourier transform of $u$).
In fact, by \cite[Theorem 3.1]{SS}, $p$ has the form
\[
p(x,\xi)=\psi_Z(\Phi^T(x)\xi), \quad x,\xi\in\BR^d,
\]
where $\psi_Z$ is the characteristic exponent of the L\'evy process $Z$. For later use, we give below more explicit form of $p$. Since
\[
 \psi_Z(\bar\xi)=-i\xi^{(1)}+\frac12(\xi^{(2)},\xi^{(2)})+\psi(\xi^{(3)}),
\quad \bar\xi=(\xi^{(1)},\xi^{(2)},\xi^{(3)})\in\BR\times\BR^d\times\BR^k,
\]
where $\psi$ is the characteristic exponent of $N$, and
\[
 \Phi^T(x)\xi=((b(x),\xi),\sigma^T(x)\xi,\gamma^T(x)\xi)^T\in\BR^{1+d+k},\quad \xi\in\BR^d,
\]
we see that
\begin{equation}
\label{eq3.20}
p(x,\xi)=-i(b(x),\xi)+\frac12(a(x)\xi,\xi)+\psi(\gamma^T(x)\xi),\quad x,\xi\in\BR^d.
\end{equation}

From the  Riesz representation theorem it follows that  there exists a family of kernels $\{p_t(x,dy), (t,x)\in(0,\infty)\times\BR^d\}$ such that
\[
P_tu(x)=\int_{\BR^d}u(y)p_t(x,dy),\quad  u\in\BB_b(\BR^d).
\]
Furthermore, for each $(t,x)\in(0,\infty)\times\BR^d$, $p_t(x,\cdot)$ is a uniquely defined positive Radon measure on $\BR^d$ (for more details see \cite[p. 7]{BSW}). In case the kernels  are absolutely continuous with respect to the Lebesgue measure on $\BR^d$, i.e. are represented as $p_t(x,dy)=p_t(x,y)\,dy$, we call $p_t(x,y)$ the transition densities of $\BX$.

\subsection{Probabilistic solutions and BSDEs, L\'evy-type operators}

We first consider the case where $L$ is a L\'evy type operator of the form (\ref{eq3.3}).
\begin{definition}
\label{def3.1}
A measurable  function $u:[0,T)\times\BR^d\rightarrow\BR$  is called a  probabilistic solution of (\ref{eq1.1}) if
\begin{equation}
\label{eq3.5}
E|\varphi(X^{s,x}_T)|+E\int^T_s|f(t,X^{s,x}_t,u(r,X^{s,x}_t))|\,dt<\infty
\end{equation}
and
\begin{equation}
\label{eq3.23}
u(s,x)=E\Big(\varphi(X^{s,x}_T)
+\int^T_sf(t,X^{s,x}_t,u(r,X^{s,x}_t))\,dt\Big)
\end{equation}
for $(s,x)\in[0,T)\times\BR^d$.
\end{definition}

Note that (\ref{eq3.23}) can be considered as a nonlinear version of the Feynman--Kac formula.

In many interesting situations Definition \ref{def3.1} is too strong, because for some natural assumptions on $\varphi,f$ we do not know whether (\ref{eq3.5}) is satisfied  for every $(s,x)$. On the other hand, for some subclasses of operators of the form (\ref{eq1.4}) we have additional information on the associated process $X^{s,x}$ which allows us to give a meaningful weaker version of the definition of a solution. One example we given below. The second will be given in the next subsection.

Suppose that the transition kernels of $\BX$ are absolutely continuous with respect to the Lebesque measure and let $p_t(x,y)$ denote the transition densities.
Then (\ref{eq3.23}) can be written as
\begin{equation}
\label{eq3.24}
u(s,x)=\int_{\BR^d}\varphi(y)p_{T-s}(x,y)\,dy
+\int^T_s\!\int_{\BR^d}f(t,y,u(t,y))p_{t-s}(x,y)\,dt\,dy.
\end{equation}
Of course, to define the  right-hand side of (\ref{eq3.24}) we need not know that $u$ is defined everywhere on $[0,T)\times\BR^d$.  This leads to the following slightly weakened version of Definition \ref{def3.1}:

\begin{definition}
A real function $u$ defined a.e. on $[0,T)\times\BR^d$ is called an integral solution of (\ref{eq1.1}) if (\ref{eq3.24}) holds for a.e. $(s,x)\in [0,T)\times\BR^d$ (in particular the integrals in (\ref{eq3.24}) are well defined 	for a.e. $(s,x)\in [0,T)\times\BR^d$).
\end{definition}

Theorem \ref{th3.2} below provides  very useful relation between probabilistic solutions of (\ref{eq1.1})  and solutions of BSDEs of the form
\begin{equation}
\label{eq3.2}
Y^{s,x}_t=\varphi(X^{s,x}_T)+\int^T_tf(r,X^{s,x}_r,Y^{s,x}_r)\,dr
-\int^T_tdM^{s,x}_r,\quad t\in[s,T].
\end{equation}

In what follows, for  $0\le s<T$ we denote by $(\FF^s_t)_{t\in[s,T]}$ the standard augmentation of the filtration $(\FF^{0,s}_t)_{t\in[s,T]}$, where $\FF^{0,s}_t=\sigma((W_r-W_s,N_r-N_s), r\in[s,t])$. By slight abuse of notation, we say that $f$ satisfies (H3) (resp. (H4)) with  $X$ replaced by $X^{s,x}$ if $E\int^T_s|f(t,X^{s,x}_t)|\,dt<\infty$  for $r>0$ (resp. $E\int^T_s|f(t,X^{s,x}_t,0)|\,dt<\infty$).

\begin{theorem}
\label{th3.2}
Assume that for every $(s,x)\in[0,T)\times\BR^d$, $E|\varphi(X^{s,x}_T)|<\infty$ and $f$ satisfies \mbox{\rm(H1), (H2)} and
\mbox{\rm(H3), (H4)} with $X$ replaced by $X^{s,x}$. Then
\begin{enumerate}[\rm(i)]
\item For every $(s,x)\in[0,T)\times\BR^d$ there exists a unique solution $(Y^{s,x},M^{s,x})$ of \mbox{\rm(\ref{eq3.2})}.

\item Let $u(s,x)=EY^{s,x}_s$, $(s,x)\in[0,T)\times\BR^d$. Then $Y^{s,x}_t=u(t,X^{s,x}_t)$ $P$-a.s for every $t\in[s,T]$ and $u$ is a probabilistic solution of \mbox{\rm(\ref{eq1.1})}.

\item Let $u$ be a probabilistic solution of \mbox{\rm(\ref{eq1.1})}. Then for every $(s,x)\in[0,T)\times\BR^d$ there exists a unique version $Y^{s,x}$
of the process $[s,T]\ni t\mapsto u(t,X^{s,x}_t)$ such that the pair $(Y^{s,x},M^{s,x})$, where $M^{s,x}$ is a  c\`adl\`ag version of the martingale
\[
\tilde M^{s,x}_t=E\Big(\varphi(X^{s,x}_T)
+\int^T_sf(r,X^{s,x}_r,u(r,X^{s,x}_r))\,dr\,\big|\,\FF^s_t\Big)
-u(s,X^{s,x}_s), \quad t\in[s,T],
\]
is the unique solution of \mbox{\rm(\ref{eq3.2})}.
\end{enumerate}
\end{theorem}
\begin{proof}
Part (i) follows straight from Theorem \ref{th2.2} applied to $\xi=\varphi(X^{s,x}_T)$ and the process $X^{s,x}$. By uniqueness of the solution of (\ref{eq3.1}), $X^{t,X^{s,x}_t}_r=X^{s,x}_r$, $r\in[t,T]$. From this and uniqueness of (\ref{eq3.2}) it folows that
\begin{equation}
\label{eq3.22}
Y^{s,x}_t=Y^{t,X^{s,x}_t}_t\quad P\mbox{-a.s.}
\end{equation}
for $t\in[s,T]$.  By (\ref{eq3.22}),
$Y^{s,x}_t=u(t,X^{s,x}_t)$ $P$-a.s for every $t\in[s,T]$, so by Fubini's theorem,
\begin{equation}
\label{eq3.26}
E\int^T_s|f(t,X^{s,x}_t,Y^{s,x}_t)-f(t,X^{s,x}_t,u(t,X^{s,x}_t))|\,dt=0.
\end{equation}
Therefore taking $t=s$ in (\ref{eq3.2}) and integrating with respect to $P$ we get (\ref{eq3.23}). This proves (ii). To prove (iii), we write
\begin{align*}
\tilde 	M^{s,x}_t&=\int^t_sf(r,X^{s,x}_r,u(r,X^{s,x}_r))\,dr\\
	&\quad+E\Big(\varphi(X^{s,x}_T)
	+\int^T_tf(r,X^{s,x}_r,u(r,X^{s,x}_r))\,dr\,\big|\,\FF^s_t\Big)
	-u(s,X^{s,x}_s).
\end{align*}
By the Markov property,
\begin{align*}
\tilde M^{s,x}_t&=\int^t_sf(r,X^{s,x}_r,u(r,X^{s,x}_r))\,dr\\
	&\quad+E\Big(\varphi(X^{s,x}_T)
	+\int^T_sf(r,X^{s,x}_r,u(r,X^{s,x}_r))\,dr\Big)\Big|_{(s,x):=(t,X^{s,x}_t)}
	-u(s,X^{s,x}_s)\\
	&=\int^t_sf(r,X^{s,x}_r,u(r,X^{s,x}_r))\,dr +u(t,X^{s,x}_t)-u(s,X^{s,x}_s).
\end{align*}
Since the martingale $\tilde M^{s,x}$ has a unique c\`adl\`ag version $M^{s,x}$, it follows from the above inequality that $t\mapsto u(t,X^{s,x}_t)$ has a unique c\`adl\`ag version $Y^{s,x}$, and moreover,
\[
Y^{s,x}_T- Y^{s,x}_t =-\int^T_tf(r,X^{s,x}_r, Y^{s,x}_r))\,dr
+ M^{s,x}_T-M^{s,x}_t,\quad t\in[s,T],\quad P\mbox{-a.s.},
\]
which shows that $(Y^{s,x},M^{s,x})$ is a solution of  (\ref{eq3.2}).
\end{proof}

\begin{corollary}
Under the assumptions of Theorem \ref{th3.2} there exists at most one probabilistic solution of \mbox{\rm(\ref{eq1.1})}.
\end{corollary}
\begin{proof}
Suppose that $u_i$, $i=1,2$, are  probabilistic solutions. Let $Y^{s,x,i}$ denote the versions of $t\mapsto u_i(t,X^{s,x}_t)$ of  Theorem \ref{th3.2}(iii). By uniqueness (see Theorem \ref{th2.2}), $Y^{s,x,1}_s=Y^{s,x,2}_s$ $P$-a.s., so $u_1(s,x)=u_2(s,x)$.
\end{proof}

\begin{corollary}
\label{cor3.5}
Let  the assumptions of Theorem \ref{th3.2} hold and  $u$ be the probabilistic solution of \mbox{\rm(\ref{eq1.1})}. Then
\begin{equation}
\label{eq3.6}
	|u(s,x)|\le e^{\mu(T-s)}E\Big(|\varphi(X^{s,x}_T)|+\int^T_s|f(t,X^{s,x}_t,0)|\,dt\Big).
\end{equation}
\end{corollary}
\begin{proof}
Let $(Y^{s,x},M^{s,x})$ be the pair of Theorem \ref{th3.2}(iii). For $t\in[s,T]$ we set
$\tilde Y^{s,x}_t=e^{\mu(t-s)}Y^{s,x}_t, \tilde M^{s,x}_t=e^{\mu(t-s)}M^{s,x}_t$. Then (see Remark \ref{rem2.5})  the pair $(\tilde Y^{s,x},\tilde M^{s,x}$) is the  unique solution of the BSDE of the form  (\ref{eq3.2}) but with final condition $e^{\mu (T-s)}\varphi(X^{s,x}_T)$ and coefficient
$\tilde f(t,x,y)=e^{\mu(t-s)}f(t,x,e^{-\mu(t-s)}y)-\mu y$. Applying the Meyer--Tanaka formula we get
\begin{align*}
|\tilde Y^{s,x}_s|&\le e^{\mu(T-s)}|\varphi(X^{s,x}_T)|
-\int^T_s\mbox{sgn}(\tilde Y^{s,x}_{t-})\,d\tilde Y^{s,x}_t\\
&\le e^{\mu(T-s)}|\varphi(X^{s,x}_T)|
+\int^T_s\mbox{sgn}(\tilde Y^{s,x}_t)\tilde f(t,X^{s,x}_t,\tilde Y^{s,x}_t)\,dt
-\int^T_s\mbox{sgn}(\tilde Y^{s,x}_{t-})\,d\tilde M^{s,x}_t.
\end{align*}
Hence
\begin{align*}
E|Y^{s,x}_s|&\le e^{\mu(T-s)}E|\varphi(X^{s,x}_T)| +E\int^T_s\mbox{sgn}(Y^{s,x}_t)(\tilde f(t,X^{s,x}_t,Y^{s,x}_t)
	-\tilde f(t,X^{s,x}_t,0))\,dt\\
	&\quad+E\int^T_s\mbox{sgn}(Y^{s,x}_t)\tilde f(t,X^{s,x}_t,0)\,dt.
\end{align*}
This shows (\ref{eq3.6}) because $\tilde f$ satisfies (H2) with $\mu=0$, so the second term on the right-hand side of the above inequality is less then or equal to zero.
\end{proof}

In the next theorem we assume that the transition kernels of $\BX$ are absolutely continuous with respect to the Lebesgue measure.

\begin{theorem}
Assume that for every $s\in[0,T)$, $E|\varphi(X^{s,x}_T)|<\infty$ for a.e. $x\in\BR^d$ and $f$ satisfies \mbox{\rm(H1), (H2)} and \mbox{\rm(H3), (H4)}
with $X$ replaced by $X^{s,x}$ for   a.e. $x$. Then
\begin{enumerate}[\rm(i)]
\item There exists an integral solution $u$ of \mbox{\rm(\ref{eq1.1})}. In fact, there exists $u$ such that for every $s\in[0,T)$ equation \mbox{\rm(\ref{eq3.24})} holds true for a.e. $x\in\BR^d$.

\item Let $u_1,u_2$ be integral solutions such that for every $s\in[0,T)$ equation \mbox{\rm(\ref{eq3.24})} holds for a.e. $x$ with $u$ replaced by $u_1$ and also with $u$ replaced by $u_2$. Then $u_1(s,\cdot)=u_2(s,\cdot)$ a.e. for $s\in[0,T)$.
\end{enumerate}
\end{theorem}
\begin{proof}
By Theorem \ref{th2.2}, for each $s\in[0,T)$ there exists a unique solution $(Y^{s,x},M^{s,x})$ of (\ref{eq3.2}) for a.e. $x$. Set $u(s,x)=EY^{s,x}_s$ for $(s,x)$ for which $Y^{s,x}$ is well defined.
Since the transition kernels of $\BX$ are absolutely continuous, (\ref{eq3.22}) holds for  $t\in[s,T)$.
Hence, for $s\in[0,T)$ and a.e. $x\in\BR^d$, we have
$Y^{s,x}_t=u(t,X^{s,x}_t)$ $P$-a.s. for $t\in[s,T]$, and consequently, for every $s\in[0,T)$  equality (\ref{eq3.26}) holds for a.e. $x\in\BR^d$. Taking now $t=s$ in (\ref{eq3.2}) and integrating with respect to $x$ we get  (i).
Fix $s\in[0,T)$. From the proof of Theorem \ref{th3.2}(iii)  it follows that for $i=1,2$ there is a set $N_i\subset\BR^d$ of Lebesque measure zero such that the processes $[s,T]\ni t\mapsto u_i(t,X^{s,x}_t)$  have c\`adl\`ag versions $\bar Y^{s,x,i}$ for a.e. $x\in\BR^d\setminus N_i$. By Theorem \ref{th2.2},
$\bar Y^{s,x,1}_s=\bar Y^{s,x,2}_s$  for $x\in\BR^d\setminus(N_1\cup N_2)$, which proves (ii).
\end{proof}

\subsection{Probabilistic solutions and BSDEs, L\'evy operators}
\label{sec3.3}

In this section, we confine ourselves to L\'evy operators, i.e we assume
that $a, b$ are constant  and $\gamma=I_d$ (i.e. $k=d$).
We shall see that in that case  probabilistic solutions are renormalized solutions or weak solutions depending on the integrability conditions imposed on $\varphi,f$. We also provide stability results for these kind of solutions.

In the case under consideration
$X^{s,x}$ is a L\'evy process starting at time $s$ from $x$ with the L\'evy triplet
$(b,a,\nu)$. Therefore the characteristic function of $X^{s,x}_t$ has the form
\[
Ee^{i(\xi,X^{s,x}_t)}= e^{i(\xi,x)}e^{-(t-s)\psi(\xi)},\quad \xi\in\BR^d,
\]
where  the characteristic exponent $\psi:\BR^d\rightarrow\BC$ is given by the celebrated L\'evy-Khintchine formula
\[
\psi(\xi)=-i(b,\xi)+\frac12(\xi,a\xi)
+\int_{\BR^d}(1-e^{i(\xi,y)}+i(\xi,y)\fch_{\{|y|\le1\}})\,\nu(dy).
\]
Note that here $\psi$ denotes the characteristic exponent of $X^{s,x}$ and not only of its pure jump part, i.e. $p(x,\xi)=\psi(\xi)$.
The generator of the semigroup $(P_t)$ generated by $\BX$ restricted to $C^{\infty}_c(\BR^d)$ can be written in the  form
\begin{equation}
\label{eq3.12}
Lu(x)=- \psi(D)u(x):=-\int_{\BR^d}e^{i(x,\xi)}\psi(\xi)\hat u(\xi)\,d\xi,\quad u\in C^{\infty}_c(\BR^d),
\end{equation}
which is the special case of (\ref{eq3.3}).
For a characteristic exponent  $\psi$ we consider the bilinear form
\[
\Psi(u,v)=\int_{\BR^d}\hat u(\xi)\overline{\hat v(\xi)}\psi(\xi)\,d\xi, \quad u,v\in H^{\psi,1}(\BR^d),
\]
where
\[
H^{\psi,1}(\BR^d)=\Big\{u\in L^2(\BR^d):\int_{\BR^d}(1+|\psi(\xi)|)|\hat u(\xi)|^2\,d\xi<\infty\Big\}.
\]

Note that the operator $L=-(-\Delta)^{\alpha/2}=:\Delta^{\alpha/2}$ considered in \cite{XCV} corresponds to $\psi(\xi)=|\xi|^{\alpha}$ and in this case the form $(\Psi,H^{\psi,1}(\BR^d))$ admits equivalent description as
\[
\Psi(u,v)=\frac{c_{d,\alpha}}{2}
	\int_{\BR^d\times\BR^d\setminus\Delta}
	\frac{(u(x)-u(y))(v(x)-v(y))}{|x-y|^{d+\alpha}}\,dx\,dy
	\]
for some constant  $c_{d,\alpha}>0$  (see, e.g., \cite{FOT} for more details).

Let $\EE$ be  the time-dependent Dirichlet form $\EE$ on $L^2(\BR\times\BR^d)$ associated with $\Psi$   in the sense described in \cite[(1.5)]{O1} or \cite[(6.1.8)]{O2} (see also \cite[(2.4)]{KR:NoD} or \cite[(2.2)]{KR:JEE}).  We define capacity Cap associated with $\EE$ and quasi-continuous functions as in \cite[Section 4]{O1} or \cite[Section 6.2]{O2}. Let $Q_T=(0,T)\times\BR^d$.
We say that some property holds quasi-everywhere (q.e. in abbreviation) in $Q_T$ if it holds in $Q_T$ except on a set whose capacity Cap equals zero.
For other notions of capacity which lead to  the same notion of exceptional subsets of $Q_T$,  i.e. subsets of capacity zero, we refer to \cite{KR:JEE}.

\begin{remark}
\label{rem3.8}
Let $\tilde B\subset Q_T$ be a Borel set.
By  remarks on p. 298 in  \cite{O1}, $\mbox{Cap}(\tilde B)=0$ if and only if
$I(\tilde B):=\int^T_0\!\!\int_{\BR^d}P(\exists t\in(0,T-s):(s+t,X^{s,x}_{s+t})\in\tilde  B)\,dx\,ds=0$. From this equivalence it follows that if the transition kernels of $\BX$ have strictly positive densities, then  for each $t\in(0,T)$ we have $I(\{t\}\times B)>0$ for any Borel set $B\subset\BR^d$ of positive Lebesque measure. Hence, if some property holds q.e. in $Q_T$, then under the above assumption on the transition kernels  it holds for a.e. $x\in \BR^d$ for each $s\in(0,T)$.
\end{remark}

\begin{definition}
\label{def3.7}
A quasi-continuous  function $u$ defined q.e.  on $Q_T$ is called a {\em generalized probabilistic solution}  of (\ref{eq1.1}) if (\ref{eq3.24}) holds for q.e. $(s,x)\in Q_T$ (in particular the integrals in (\ref{eq3.24}) are well defined 	for q.e. $(s,x)\in Q_T$).
\end{definition}

If the transition kernels of $\BX$ have strictly positive densities, then by Remark \ref{rem3.8}, generalized probabilistic solutions are integral solutions.

\begin{theorem}
\label{th3.9}
Let $E|\varphi(X^{s,x}_T)|<\infty$ for q.e. $(s,x)\in Q_T$, $f$ satisfy \mbox{\rm(H1), (H2)} and  \mbox{\rm(H3), (H4)} with $X$ replaced by $X^{s,x}$ for q.e. $(s,x)\in Q_T$. Then for q.e. $(s,x)\in Q_T$ there exists a unique generalized probabilistic solution of \mbox{\rm(\ref{eq1.1})}. It is unique in the sense that if  $u_1,u_2$ are two solutions, then $u_1=u_2$ q.e. on $Q_T$. Moreover, for q.e. $(s,x)\in Q_T$ there exists a unique solution $(Y^{s,x},M^{s,x})$ of \mbox{\rm(\ref{eq3.2})} and $Y^{s,x}_s=u(s,x)$ a.s. for q.e. $(s,x)\in Q_T$.
\end{theorem}
\begin{proof}
This is a special case of \cite[Theorem 5.8]{K:JFA}.
\end{proof}

In many cases, one can show that probabilistic or generalized probabilistic solutions coincide with other notions of solutions considered in the theory of PDEs, for instance there are  viscosity solutions, weak solutions or renormalized solutions. As for viscosity solutions, we refer the reader to \cite{BBP}. Weak and renormalized solutions will be addressed in Section \ref{sec5}.

\section{Stability of probabilistic solutions}
\label{sec4}

For $n\ge1$ let $\sigma^n:\BR^d\rightarrow\BR^{d\times k}$, $b^n:\BR^d\rightarrow\BR^d$,  $\gamma^n:\BR^d\rightarrow\BR^{d\times k}$ be bounded Lipschitz continuous functions  and  $\nu_n$ be a  L\'evy measure on $\BR^d$. We set $a^n=\sigma^n\cdot(\sigma^n)^T$ and consider operators $L^n$ with symbols
\[
p_n(x,\xi)=-i(b^n(x),\xi)+\frac12(a^n(x)\xi,\xi)+\psi((\gamma^n)^T(x)\xi),\quad x,\xi\in\BR^d.
\]
In the special but important case where $d=k$ and $\gamma_n$ is the $d$-dimensional identity matrix $L^n$ has the form
\begin{align*}
L^nu_n(x)&=\sum^d_{i=1}b^n_i(x)\partial_{x_i}u_n(x)
+\frac{1}{2}\sum^d_{i,j=1}a^n_{ij}(x)\partial^2_{x_ix_j}u_n(x) \nonumber\\
&\quad+\int_{\BR^d}(u_n(x+y)-u_n(x)-y\cdot\nabla u_n(x)
\fch_{|y|\le1})\,\nu_n(dy).
\end{align*}
We denote by  $u_n$  the  probabilistic solution of the problem
\begin{equation}
	\label{eq4.3}
	\begin{cases}
		\partial_su_n+L^nu_n=-f(s,x,u_n),\quad (s,x)\in Q_T:=(0,T)\times\BR^d,\\
		u_n(T,x)=\varphi(x),\quad x\in\BR^d.
	\end{cases}
\end{equation}
It exists and is unique by Theorem \ref{th3.2}.

Since $N$ is  a pure-jump symmetric L\'evy process in $\BR^k$  with the characteristics $(0,0,\nu)$, the L\'evy--Khinchine formula says that the characteristic function of $N_t-N_s$ is of the form
\[
Ee^{i(\xi,N_t-N_s)}=e^{-(t-s)\psi(\xi)},\quad \xi\in\BR^k,
\]
for some  $\psi:\BR^k\rightarrow\BR$ called the characteristic exponent of $N$ (see Section \ref{sec5} for more details). Suppose that $N^n$ is another pure-jump  L\'evy processes. In what follows we denote its  characteristic exponent  by $\psi_n$.

\subsection{Stability  of solutions of SDEs}

We begin with  some stability results for solutions of SDEs of the form (\ref{eq3.18}). We will need the following assumptions:
\begin{equation}
\label{eq4.6}
\begin{cases}
a^n,b^n,\gamma^n\mbox{ are uniformly bounded}, \\
\sigma^n\rightarrow\sigma, \,b^n\rightarrow b,\,\gamma^n\rightarrow\gamma\mbox{ locally uniformly on }\BR^d
\end{cases}
\end{equation}
and
\begin{equation}
\label{eq4.4}
\lim_{n\rightarrow\infty}|\psi_n(\xi)-\psi(\xi)|=0,\quad \xi\in\BR^k.
\end{equation}

\begin{proposition}
\label{prop3.5}
Let $N,N^n$, $n\ge1$, be pure-jump  L\'evy processes and $W$ be a standard Wiener process independent of the processes $N,N^n$, $n\ge1$. Assume that
\mbox{\rm(\ref{eq4.4})} is satisfied.
Then there exists a probability space and processes $\tilde N,\tilde N^n$, $\tilde W^n$, $n\ge1$, defined on it such that $\tilde W^n$ is a standard Wiener process independent of $\tilde N,\tilde N^n$ for each $n\ge1$, and moreover,
\begin{equation}
\label{eq3.11}
(W,N)=_d(\tilde W,\tilde N),\qquad (W,N^n)=_d (\tilde W^n,\tilde N^n)
\end{equation}
and
\begin{equation}
\label{eq3.14}
(\tilde W^n, \tilde N^{n})\arrowpl(\tilde W,\tilde N)\mbox{ in } D([0,T];\BR^{d+k}),
\qquad \tilde\BF^n\arrowwl\tilde\BF\mbox{ on }[0,T],
\end{equation}
where $\tilde\BF$ (resp. $\tilde \BF^{n}$) is the standard augmentation of the natural filtration generated by $(\tilde W,\tilde N)$ (resp. $(\tilde W^n,\tilde N^{n})$).
\end{proposition}
\begin{proof}
Condition (\ref{eq4.4}) implies that $\lim_{n\rightarrow\infty}
\sup_{s\le t\le q} |e^{-(t-s)\psi_n(\xi)}-e^{-(t-s)\psi(\xi)}|=0$ for all $\xi\in\BR^d,q>0$, so by \cite[Corollary 2]{JS}, $N^{n}\arrowd N$ in $D([0,T];\BR^k)$.
Hence $(W,N^{n})\arrowd (W,N)$ in $D([0,T];\BR^{d+k})$. By Skorokhod's representation theorem
(see, e.g., \cite[Theorem 4.30]{K}),  there is a probability space with processes
$(\tilde W,\tilde N)$, $(\tilde W^n,\tilde N^{n})$  defined on it such
that (\ref{eq3.11}) and the first part of (\ref{eq3.14}) are satisfied. The processes $(W,N^{n})$ have independent increments, so the second part of (\ref{eq3.14}) follows from \cite[Proposition 2]{CMS}. 
\end{proof}

Let $\tilde W,\tilde W^n$ and $\tilde N,\tilde N^n$ be the processes of Proposition \ref{prop3.5}. For $(s,x)\in[0,T)\times\BR^d$ we denote by $\tilde X^{s,x}$ be the unique solution of  (\ref{eq3.1}) with $W,N$ replaced by $\tilde W,\tilde N$,  and by
$\tilde X^{s,x,n}$  the unique strong solution of the SDE
\[
\tilde X^{s,x,n}_t=x+\int^t_s\sigma^n(\tilde X^{s,x,n}_r)\,d\tilde W^n_r+\int^t_sb^n(X^{s,x,n}_r)\,dr
+\int^t_s\gamma^n(X^{s,x}_{r-})\,d\tilde N^{n}_r,\quad t\in[s,T].
\]

We can now state a useful  sufficient condition for hypotheses (C) to hold with $X^n$ repaced by $\tilde X^{s,x,n}$ and $X$ replaced by $\tilde X^{s,x}$.

\begin{proposition}
\label{prop4.2}
Assume that  \mbox{\rm (\ref{eq4.6}), (\ref{eq4.4})} are satisfied. Then
\begin{equation}
\label{eq4.7}
\tilde X^{s,x,n}\arrowpl\tilde X^{s,x}\quad\mbox{in }D([0,T];\BR^d).
\end{equation}
\end{proposition}
\begin{proof}
	Follows from  the first part of (\ref{eq3.14}) and \cite[Corollary 3]{S}.
\end{proof}

\subsection{Stability of solutions of PDEs}

We first give a general remark concernig applications of results of Section \ref{sec2} to the convergence of solutions of PDEs.

Fix $(s,x)\in[0,T)\times\BR^d$. Define $\tilde X^{s,x},\tilde X^{s,x,n}$ as in Proposition \ref{prop4.2} and denote by $\tilde Y^{s,x}$ (resp. $\tilde Y^{s,x,n}$)  the first component of the solution of  BSDE$(\varphi(\tilde X^{s,x}_T),\tilde X^{s,x},f,\tilde\BF)$ (resp. BSDE$(\varphi(\tilde X^{s,x,n}_T),\tilde X^{s,x,n},f,\tilde\BF^n)$).
By uniqueness, we have  
\begin{equation}
\label{eq4.23}
X^{s,x}=_d\tilde  X^{s,x},\qquad
X^{s,x,n}=_d\tilde  X^{s,x,n},\quad n\ge1.
\end{equation}
It follows that the probabilistic solution $u$ of (\ref{eq1.1}) satisfies the equation
\begin{equation}
	\label{eq4.8}
	u(s,x)=E\Big(\varphi(\tilde X^{s,x}_T)
	+\int^T_sf(t,\tilde X^{s,x}_t,u(t,\tilde X^{s,x}_t))\,dt\Big)
\end{equation}
and similarly, the probabilistic solution $u_n$ of (\ref{eq4.3}) satisfies
\begin{equation}
	\label{eq4.9}
	u_n(s,x)=E\Big(\varphi(\tilde X^{s,x,n}_T)
	+\int^T_sf(t,\tilde X^{s,x,n}_t,u_n(t,\tilde X^{s,x,n}_t))\,dt\Big).
\end{equation}
Since $E\tilde Y^{s,x}_s$ equals the right-hand side of (\ref{eq4.8}) and
$E\tilde Y^{s,x,n}_s$ equals the right-hand side of (\ref{eq4.9}), we see that pointwise convergence of $u_n$ to $u$ will be proved once we prove that
\begin{equation}
	\label{eq4.10}
	\tilde Y^{s,x,n}_s\arrowdl\tilde Y^{s,x}_s\quad\mbox{in }\BR.
\end{equation}
Similarly, the convergence of $u_n$ to $u$ in other topologies is related to (\ref{eq4.10}). Such idea of proving stability of solutions, but for linear equations and Neumann problem,  appeared in \cite{RS}.

To show (\ref{eq4.10}) we will use Theorem \ref{th2.4} with the  data
\begin{equation}
	\label{eq4.11}
	X=\tilde X^{s,x},\quad  X^n=\tilde X^{s,x,n},\quad \xi=\varphi(\tilde X^{s,x}_T),\quad \xi^n=\varphi(\tilde X^{s,x,n}_T).
\end{equation}

Below we present some sufficient conditions which together with (\ref{eq4.7})  imply   (H5)--(H7). As for $f$, we will assume that
\begin{equation}
\label{eq4.12}
|f(t,x,y)-f(t,x,0)|\le h(t,x)H(|y|)
\end{equation}
for some nonnegative measurable $h$  and nondecreasing $H:\BR_+\rightarrow\BR$. Examples of $f$ satisfying (H1), (H2) and (\ref{eq4.12}) are functions defined by (\ref{eq1.8}).
In the first result we require no additional assumptions on the underlying processes. The price is that we require that $\varphi$ is bounded and continuous and our assumptins on $f$, when applied to  (\ref{eq1.8}), say that $h_1,h_2$ are bounded.  Next we show a result  which covers integrable $\varphi$ and $f$ of the form (\ref{eq1.8}) with integrable $\varphi,h_1$ and essentially bounded $h_2$.

\begin{theorem}
\label{th4.3}
Assume that  $\varphi$ is bounded continuous, $f$ satisfies   \mbox{\rm(H1), (H2), (\ref{eq4.12})} with bounded $h$ and moreover  $f(\cdot,\cdot,0)$ is  bounded.
Let $u$ be the probabilistic solution of \mbox{\rm(\ref{eq1.1})} and $u_n$, $n\ge1$, be the probabilistic solution of \mbox{\rm(\ref{eq4.3})}.
If \mbox{\rm(\ref{eq4.6}), (\ref{eq4.4})} are satisfied, then
$u_n(s,x)\rightarrow u(s,x)$ for every $(s,x)\in[0,T)\times\BR^d$.
\end{theorem}
\begin{proof}
Fix $(s,x)\in[0,T)\times\BR^d$ and define $\xi^n,\xi$ and  $\tilde X^n,\tilde X$ by (\ref{eq4.11}).
Hypothesis (C) is satisfied by Proposition \ref{prop4.2}.
Clearly,  if $\varphi$ is bounded and continuous, then (H7) is satisfied and (\ref{eq4.12}) together with boundedness of $f(\cdot,\cdot,0)$ imply (H5), (H6).  Therefore (\ref{eq4.10}) follows from Theorem \ref{th2.4}, which proves the theorem.
\end{proof}

We shall see in Section \ref{sec6} that the following theorem, in which we impose additional assumptions on the underlying processes, is applicable in many interesting situations.

\begin{theorem}
\label{th4.4}
Let \mbox{\rm(\ref{eq4.6}), (\ref{eq4.4})} be satisfied and the transition kernels be absolutely continuous with respect to the Lebesgue measure on $\BR^d$. Assume that $\varphi\in L^q(\BR^d)$ for some $q\in(1,\infty]$, $f$ satisfies
\mbox{\rm(H1), (H2), (\ref{eq4.12})} with $h\in L^{\infty}(Q_T)$.	
Let $u$ be the integral solution of \mbox{\rm(\ref{eq1.1})} and $u_n$, $n\ge1$, be the integral solution of \mbox{\rm(\ref{eq4.3})}.
\begin{enumerate}[\rm(i)]
\item   If $f(\cdot,\cdot,0)\in L^{\infty}(Q_T)$ and the corresponding densities satisfy the assumption: for each $t\in(0,T]$,
\begin{equation}
\label{eq4.14}
\sup_{x,y\in\BR^d}p_t(x,y)<\infty,\qquad  \sup_{n\ge N} \sup_{x,y\in\BR^d}p^n_t(x,y)<\infty\quad\mbox{for some }N\ge1,
\end{equation}
then  $u_n(s,x)\rightarrow u(s,x)$ for every $(s,x)\in[0,T)\times\BR^d$.
			
\item If $f(\cdot,\cdot,0)\in L^q(Q_T)$ for some $q>1$ and
\begin{equation}
\label{eq4.15}
\int_{\BR^d}p_t(x,y)\,dx\le1,\qquad \int_{\BR^d}p^n_t(x,y)\,dx\le1,\quad t\in(0,T],\, y\in\BR^d,
\end{equation}
 then for each $s\in[0,T)$,  $u_n(s,\cdot)\rightarrow u(s,\cdot)$ in $L^1_{loc}(\BR^d)$.		
\end{enumerate}
\end{theorem}
\begin{proof}
The proof of part (i) goes as the proof of Theorem \ref{th4.3}. The only difference is that we have to show (H7) with $\xi,\xi^n$ defined by (\ref{eq4.11})  and $\varphi\in L^q(\BR^d)$.
Let $\varphi\in L^q(\BR^d)$ with $q\in(1,\infty)$. By (\ref{eq4.23}), $E\varphi(X^{s,x}_T)=E\varphi(\tilde X^{s,x}_T)$, $E\varphi(X^{s,x,n}_T)=E\varphi(\tilde X^{s,x,n}_T)$, $n\ge1$.
To show that
\begin{equation}
	\label{eq4.16}
	\lim_{n\rightarrow\infty}E\varphi(X^{s,x,n}_T)=E\varphi(X^{s,x}_T)
\end{equation}
we choose a sequence
$\{\varphi_k\}$ of bounded continuous functions on $\BR^d$ such that $\varphi_k\rightarrow\varphi$ in $L^q(\BR^d)$. Clearly,
\begin{align*}
	E|\varphi(X^{s,x,n}_T)-\varphi(X^{s,x}_T)|&\le  E|\varphi(X^{s,x,n}_T)-\varphi_k(X^{s,x,n}_T)| +E|\varphi_k(X^{s,x,n}_T)-\varphi_k(X^{s,x}_T)| \\
	&\quad+ E|\varphi_k(X^{s,x}_T)-\varphi(X^{s,x}_T)|=:I^{n,k}_1+I^{n,k}_2+I^k_3.
\end{align*}
By H\"older's inequality and (\ref{eq4.14}), for $n\ge N$ we have
\begin{align*}
	I^{n,k}_1\le (E|\varphi-\varphi_k|(X^{s,x,n}_T)|^q)^{1/q}
	&=\Big(\int_{\BR^d}|\varphi-\varphi_k|^q(y)p^n_{T-s}(x,y)\,dy\Big)^{1/q}\\
	&\le C^{1/q}_{T-s}\|\varphi-\varphi_k\|_{L^q(\BR^d)}
\end{align*}
with some constant $C_{T-s}$ independent of $n$.
Hence $\lim_{k\rightarrow\infty}I^{n,k}_1=0$  uniformly in $n\ge1$.
Furthermore, by Proposition \ref{prop4.2}, $\varphi_k(\tilde X^{s,x,n}_T)\rightarrow\varphi_k(\tilde X^{s,x}_T)$ in probability as $n\rightarrow\infty$ and $\{\varphi_k(X^{s,x,n}_T)\}_{n\ge1}$ are uniformly integrable because
\[
E|\varphi_k(X^{s,x,n}_T)|^q=\int_{\BR^d}|\varphi_k(y)|^qp^n_{T-s}(x,y)\,dy
\le C_{T-s}\|\varphi_k\|^q_{L^q(\BR^d)}.
\]
Consequently, $\lim_{n\rightarrow\infty}I^{n,k}_2=0$ for each $k\ge1$. Clearly, we also have
$\lim_{k\rightarrow\infty}I^k_3=0$. As a result, (\ref{eq4.16}) holds true.	The prove in case $q=\infty$ is similar, so we omit it.

To show (ii), we define  $f^k$ by (\ref{eq2.14}) and consider the sequence $\{\varphi_k\}$ from the proof of part (i). Let $u^k$ (resp. $u^k_n$) be the probabilistic solution of (\ref{eq1.1}) (resp. (\ref{eq4.3})) with $\varphi$ replaced by $\varphi_k$ and $f$ replaced by $f^k$.
By part (i), $u^k_n\rightarrow u^k$ pointwise as $n\rightarrow\infty$, and by (\ref{eq3.6}), the functions  $u^k_n$ are bounded uniformly in $n$. Hence
\begin{equation}
\label{eq4.13}
u^k_n(s,\cdot)\rightarrow u^k(s,\cdot)\quad\mbox{in }L^1_{loc}(\BR^d).
\end{equation}
By (\ref{eq3.6}), we also have
\[
|u^k(s,x)|^q\le 2^qe^{\mu qT}E\Big(|\varphi_k(X^{s,x}_T)|^q
+(T-s)^{q-1}\int^T_s|f^k(t,X^{s,x}_t,0)|^{q}\,dt\Big).
\]
Integrating both sides of the above inequality with respect to $x$  and using (\ref{eq4.15}) yields
\begin{equation}
\label{eq4.19}
\int_{\BR^d}|u^k(s,x)|^q\,dx\le 2^q2^{\mu q T}(\|\varphi_k\|^q_{L^q(\BR^d)}
+T^{q-1}\|f^k(\cdot,\cdot,0)\|^q_{L^q(Q_T)}).
\end{equation}
It follows in particular that for every $s\in[0,T)$  the functions $u^k(s,\cdot)$ are uniformly integrable on compact subsets of $\BR^d$.
By H\"older's inequality,
\begin{equation}
\label{eq4.20}
E\int^T_sf(t,X^{s,x}_t,0)|\,dt
\le(T-s)^{(q-1)/q}\Big(E\int^T_s|f(t,X^{s,x}_t,0|^q\,dt\Big)^{1/q}.
\end{equation}
On the other hand, by  (\ref{eq4.15}) and Fubini's theorem,
\begin{equation}
\label{eq4.21}
\int_{\BR^d}\Big(E\int^T_s|f(t,X^{s,x}_t,0)|^q\,dt\Big)dx
\le\int^T_s\!\!\int_{\BR^d}|f(t,y,0)|^q\,dt\,dy,
\end{equation}
which is finite since $f(\cdot,\cdot,0)\in L^q(Q_T)$. Therefore for each $s\in[s,T)$ condition (H4)  is satisfied with $X^{s,x}$ for a.e. $x\in\BR^d$.
Therefore for each $s\in[0,T)$ there exists a unique solution of
BSDE$(\varphi(\tilde X^{s,x}_T),\tilde X^{s,x},f,\tilde\BF)$ for a.e $x\in\BR^d$. W denote by $\tilde Y^{s,x}$ its first component and we set
$u(s,x)=E\tilde Y^{s,x}_s$ for $(s,x)$ for which $\tilde Y^{s,x}$ is well defined
From the first part of (\ref{eq2.12}) applied to $\tilde Y^{s,x}$ and the
first component of the solution of BSDE$(\varphi_k(\tilde X^{s,x}_T),\tilde X^{s,x},f^k,\tilde\BF)$ it follows  
that for every $s\in[0,T)$,
\begin{equation}
\label{eq4.18}
u^k(s,x)\rightarrow u(s,x),\qquad u^k(s,x)\rightarrow E\Big(\varphi(X^{s,x}_T)+\int^T_sf(t,X^{s,x}_t,\tilde Y^{s,x}_t)\,dt\Big)
\end{equation}
for a.e. $x\in\BR^d$. But from the assumption that the transition kernels are absolutely continuous it follows that $\tilde Y^{s,x}_t=\tilde Y^{t,\tilde X^{s,x}_t}_t$ $P$-a.s. for $t\in[s,T)$. Hence, for $s\in[0,T)$ and a.e. $x\in\BR^d$ we have
$\tilde Y^{s,x}_t=u(t, \tilde X^{s,x}_t)$ $P$-a.s. for $t\in[s,T]$. As a result, in  (\ref{eq4.18}) we may replace $\tilde Y^{s,x}_t$ by $u(t,\tilde X^{s,x}_t)$, which together with (\ref{eq4.23}) shows that $u$ is the integral solution of (\ref{eq1.1}).
The first part of (\ref{eq4.18}) and  (\ref{eq4.19}) now imply  that
\begin{equation}
\label{eq3.8}
u^k(s,\cdot)\rightarrow u(s,\cdot)\quad\mbox{in }L^1_{loc}(\BR^d),
\end{equation}
As in the proof of  (\ref{eq2.15}), we obtain
\begin{align*}
	|u^k_n(s,x)-u_n(s,x)|&\le E|(\varphi_k-\varphi)(X^{s,x,n}_T)|\\
	&\quad+ E\int^T_s|f(t,X^{s,x,n}_t,0)|\fch_{(k,\infty)}(|f(t,X^{s,x,n}_t,0)|)\,dt.
\end{align*}
By this and (\ref{eq4.15}),  we have
\begin{align*}
	&\int_{\BR^d}|u^k_n(s,x)-u_n(s,x)|^q\,dx \nonumber\\
	&\qquad\le2^q\Big(\|\varphi_k-\varphi\|_{L^q(\BR^d)} +(T-s)^{q-1}\int^T_s\!  \int_{\BR^d}|f(t,y,0)|^q\fch_{\{|f(t,y,0)|>k\}}\,dt\,dy\Big).
\end{align*}
Hence, if $f(\cdot,\cdot,0)\in L^q(Q_T)$, then for each $s\in[0,T)$,
\begin{equation}
	\label{eq3.9}	
	\lim_{k\rightarrow\infty}\sup_{n\ge1}
	\|u^k_n(s,\cdot)-u_n(s,\cdot)\|_{L^q(\BR^d)}=0.
\end{equation}
Combining  (\ref{eq4.13}) with (\ref{eq3.8}) and (\ref{eq3.9}) we get (ii).
\end{proof}

\section{Weak and renormalized solutions}
\label{sec5}

In this section, as in Subsection \ref{sec3.3}, we confine ourselves to L\'evy operators. Recall that the classic  Hartman--Wintner  condition
\begin{equation}
\label{eq5.1}
\lim_{|\xi|\rightarrow\infty}
\frac{\mbox{Re}\,\psi(\xi)}{\ln(1+|\xi|)}=\infty
\end{equation}
is sufficient for the existence of transition densities of $\BX$. Furthermore, if (\ref{eq5.1}) is satisfied, then
\[
p_t(x,y)=(2\pi)^{-d/2}\int_{\BR^d}e^{-i(x-y,\xi)}e^{-t\psi(\xi)}\,d\xi.
\]
In particular, for every $t>0$,
\begin{equation}
	\label{eq5.2}
	\sup_{x,y\in\BR^d}p_t(x,y)\le(2\pi)^{-d/2}
	\int_{\BR^d}e^{-t\mathrm{Re}\,\psi(\xi)}\,d\xi,
\end{equation}
so the first part of (\ref{eq4.14}) is satisfied. The first part of (\ref{eq4.15}) is also satisfied since $\hat p_t(x,y):=p_t(x,y)$ is the transition density of the L\'evy process with the L\'evy triplet $(-b,a,\nu)$.
From the above it follows in particular, that if (\ref{eq5.1}) is satisfied and $\varphi\in L^q(\BR^d)$, $f(\cdot,\cdot,0)\in L^q(Q_T)$ for some $q\in[1,\infty)$, then
\begin{equation}
\label{eq5.3}
E|\varphi(X^{s,x}_T)|<\infty,\,\,\,(s,x)\in Q_T,\quad
E\int^T_s|f(t,X^{s,x}_t,0)|\,dt<\infty\,\mbox{ q.e. }(s,x)\in Q_T.
\end{equation}
Indeed, the first part of (\ref{eq5.3}) is a consequence of (\ref{eq4.14}) and H\"older's inequality.
As for the second part of  (\ref{eq5.3}), we note that by (\ref{eq4.20}) and (\ref{eq4.21}), $R^{0,T}|f(\cdot,\cdot,0)|<\infty$ a.e. in $Q_T$, where $R^{0,T}f(\cdot,\cdot,0)(s,x):=E\int^T_s|f(t,X^{s,x}_t,0)|\,dt$.  In fact,
$R^{0,T}|f(\cdot,\cdot,0)|<\infty$ q.e. in $Q_T$, because by \cite[Proposition 3.4]{K:JFA},  $R^{0,T}|f(\cdot,\cdot,0)|$ is quasi-continuous.

\subsection{Weak solutions}

Let $V=H^{\psi,1}(\BR^d)$, $\WW=\{u\in L^2(0,T;V):\partial_tu\in L^2(0,T;V')\}$, where $V'$ is the dual space of $V$. We denote by $\langle\cdot,\cdot\rangle$ the duality pairing between $V'$ and $V$ and by $D(0,T;L^2(\BR^d))$ the set of Borel functions on $(0,T]\times\BR^d$ such that the mapping $(0,T]\ni t\mapsto u(t):=u(t,\cdot)\in L^2(\BR^d)$ is c\`adl\`ag, i.e. right continuous with left limits.   

\begin{definition}
$u\in L^2(0,T; V)\cap D(0,T;L^2(\BR^d))$  is called a weak solution of (\ref{eq1.1}) (with $L$ defined by (\ref{eq3.12})) if for any $t\in(0,T]$ and $\eta\in\WW$,
	\begin{align*}
		&(u(t),\eta(t))_{L^2}+\int^T_t\langle\partial_s\eta(s),u(s)\rangle\,ds
		+\int^T_t\Psi(u(s),\eta(s))\,ds\\
		&\qquad =(\varphi,\eta(T))_{L^2(\BR^d)}
		+\int^T_t\!\!\int_{\BR^d}\eta(s,x)f(s,x,u(s,x))\,ds\,dx.
	\end{align*}
\end{definition}

\begin{proposition}
\label{prop5.3}
Assume that  $\varphi\in L^2(\BR^d)$,   $f$ satisfies  \mbox{\rm(H1)--(H2)}, $f(\cdot,\cdot,0)\in L^{2}(Q_T)$ and  \mbox{\rm(\ref{eq4.12})} is satisfied with $h\in L^{\infty}(Q_T)$ and $g(|y|)=|y|$. Moreover, assume that $\psi$ satisfies the Hartman--Wintner condition \mbox{\rm(\ref{eq5.1})}. Then the generalized probabilistic solution $u$ of \mbox{\rm(\ref{eq1.1})} is a weak solution.
\end{proposition}
\begin{proof}
By Theorem \ref{th3.9} and the remarks following (\ref{eq5.2})  there exists a unique generalized solution of (\ref{eq1.1}) and $u(s,x)=Y^{s,x}_s$ a.s. for q.e. $(s,x)\in Q_T$, where $(Y^{s,x},M^{s,x})$ is the unique solution of
(\ref{eq3.2}). Therefore we can repeat the proof of Corollary \ref{cor3.5} to get (\ref{eq3.6}) for q.e. $(s,x)$, and hence, by Remark \ref{rem3.8}, for a.e. $x\in\BR^d$ for each $s\in (0,T)$.
Therefore, for each $s\in(0,T)$ we have
\[
|u(s,x)|^2\le 2e^{2\mu T}E\Big(|\varphi(X^{s,x}_T)|^2 +(T-s)\int^T_s|f(t,X^{s,x}_t,0)|^2\,dt\Big)
\]
for a.e. $x\in\BR^d$. Integrating with respect to $x$ and using  (\ref{eq4.15}) we get
\begin{equation}
\label{eq5.4}
\|u(s)\|^2_{L^2(\BR^d)}\le2e^{2\mu T}(\|\varphi\|^2_{L^2(\BR^d)}
+(T-s)\|f(\cdot,\cdot,0)\|^2_{L^2((s,T)\times\BR^d)})
	\end{equation}
for $s\in(0,T)$. It follows that if $f(\cdot,\cdot,0)\in L^2(Q_T)$ and (\ref{eq4.12}) with $g(|y|)=|y|$ is satisfied, then   $f(\cdot,\cdot,u)\in L^2(Q_T)$. Hence, by a general result proved in  \cite[Theorem 3.7]{K:JFA}, $u$ is a weak solution of (\ref{eq1.1}).
\end{proof}

\subsection{Renormalized solutions}

In what follows $T_k$, $k>0$,  is the truncation operator defined by $T_k(y)=((-k)\vee y)\wedge k$ and $\MM_{0,b}(Q_T)$ is the space of signed Borel measures on $Q_T$ of finite total variation which do not charge Borel measures of capacity Cap zero. We denote by $\|\mu\|_{TV}$ the total variation norm of a bounded (signed) Borel measure $\mu$ on $Q_T$.

As in \cite{K:JFA} we call  a  Borel measurable function $u$ defined q.e. on $Q_T$ quasi-c\`adl\`ag if for q.e. $(s,x)\in  Q_T$ the proces $[s,T]\ni t\mapsto u(t,X^{s,x}_t)$ is  c\`adl\`ag.

Following \cite{KR:NoD} we adopt the following definition of renormalized solution of (\ref{eq1.1}). In the case of local operators, this is essentially
\cite[Definition 4.1]{PPP}.

\begin{definition}
We say that a quasi-c\`adl\`ag  $u$ defined q.e. on $Q_T$  is a renormalized solution of (\ref{eq1.1}) if
\begin{enumerate}
\item[(a)] $f_u:=f(\cdot,\cdot,u(\cdot,\cdot))\in L^1(Q_T;dt\otimes dx)$ and $T_ku\in L^2(0,T;V)$ for every $k>0$,
		
\item[(b)] There exists a sequence $\{\nu_k\}\subset\MM_{0,b}(Q_T)$ such that
$\|\nu_{k}\|_{TV}\rightarrow0$ as $k\rightarrow\infty$ and for
every $k\in\BN$ and every  bounded quasi-continuous  $\eta\in\WW$ such that $\eta(0)=0$ we have
\begin{align*}
&\int^T_0\langle\partial_t \eta(t),T_ku(t)\rangle\,dt +\int^T_0\Psi(T_ku(t),\eta(t))\,dt \\
&\qquad=(T_k\varphi, \eta(T))_{L^2(\BR^d)} +\int_{Q_T}\eta f_u\,dt\,dx
+\int_{Q_T}\eta\,d\nu_k.
\end{align*}
	\end{enumerate}
\end{definition}

Note that by \cite[Corollary 4.8]{KR:NoD}, if (H2) is satisfied, then there exists at most one renormalized solution of (\ref{eq1.1}).

\begin{proposition}
\label{prop5.6}
Assume that  $\varphi\in L^1(\BR^d)$,   $f$ satisfies \mbox{\rm (H1)--(H2)},
$f(\cdot,\cdot,0 )\in L^1(Q_T)$ and  \mbox{\rm(\ref{eq4.12})} is satisfied with $h\in L^{\infty}(Q_T)$. Moreover, assume that
$\psi$ satisfies the Hartman--Wintner condition \mbox{\rm(\ref{eq5.1})}. Then the generalized probabilistic solution $u$ of \mbox{\rm(\ref{eq1.1})} is a renormalized solution. 	
\end{proposition}
\begin{proof}
By the argument from the beginning of the proof of Proposition \ref{prop5.3}, for each $s\in (0,T)$ and a.e. $x\in\BR^d$ there exists the unique solution $(Y^{s,x},M^{s,x})$ of (\ref{eq3.2}) and the uniqe generalized solution has the representation $u(s,x)=Y^{s,x}_s$ a.s.
For $(s,x)$ for which the solution of (\ref{eq3.2}) exists  define $(\tilde Y^{s,x},\tilde M^{s,x})$ as in the proof of Corollary \ref{cor3.5}.
Since $\tilde f$ satisfies (H2) with $\mu=0$ we have
\begin{align*}
\int^T_s|\tilde f(t,X^{s,x}_t,\tilde Y^{s,x}_t)|\,dt
&\le \int^T_s|\tilde f(t,X^{s,x}_t,\tilde Y^{s,x}_t)-\tilde f(t,X^{s,x}_t,0)|\,dt
+\int^T_s|\tilde f(t,X^{s,x}_t,0)|\,dt\\
&=-\int^T_s\mbox{sgn}(\tilde Y^{s,x}_t)(\tilde f(t,X^{s,x}_t,\tilde Y^{s,x}_t)-\tilde f(t,X^{s,x}_t,0))\,dt\\
&\quad+\int^T_s|\tilde f(t,X^{s,x}_t,0)|\,dt.
\end{align*}
By the Meyer--Tanaka formula,
\begin{align*}
E|\tilde Y^{s,x}_T|-E|\tilde Y^{s,x}_s|\ge E\int^T_s\mbox{sgn}(\tilde Y^{s,x}_{t-})\,d\tilde Y^{s,x}_t
		=-E\int^T_s\mbox{sgn}(\tilde Y^{s,x}_{t-})\tilde f(t,X^{s,x}_t,\tilde Y^{s,x}_t))\,dt.
	\end{align*}
	Hence
	\begin{align*}
	E\int^T_s|\tilde f(t,X^{s,x}_t,\tilde Y^{s,x}_t)|\,dt&\le e^{\mu(T-s)}E|\varphi(X^{s,x}_T)|+2E\int^T_s|\tilde f(t,X^{s,x}_t,0)|\,dt\\
	&\le e^{\mu(T-s)}E|\varphi(X^{s,x}_T)|
	+2E\int^T_se^{\mu(t-s)}|f(t,X^{s,x}_t,0)|\,dt.
	\end{align*}
On the other hand,
\[
E\int^T_s|f(t,X^{s,x}_t,Y^{s,x}_t)|\,dt\le E\int^T_se^{-\mu(t-s)}|\tilde f(t,X^{s,x}_t,\tilde Y^{s,x}_t)|\,dt +\mu E\int^T_s|Y^{s,x}_t|\,dt.
\]
Combining the above two estimates and then taking  $s=0$ and integrating with respect to $x$  we get	
\begin{equation}
\label{eq5.5}
\|f_u\|_{L^1(Q_T)}\le e^{\mu T}(\|\varphi\|_{L^1(\BR^d)}
+2\|f(\cdot,\cdot,0)\|_{L^1(Q_T)})+\mu \|u\|_{L^1(Q_T)}.
	\end{equation}
Furthermore, since for each $s\in(0,T)$ estimate  (\ref{eq3.6}) holds for a.e. $x\in\BR^d$, we see that
\begin{equation}
\label{eq5.7}
\|u(s)\|_{L^1(\BR^d)}\le e^{\mu(T-s)}(\|\varphi\|_{L^1(\BR^d)}+\|f(\cdot,\cdot,0)\|_{L^1(Q_T)})
\end{equation}
for $s\in(0,T)$. Hence $f_u\in L^1(Q_T)$.
That $u$ is a renormalized solution now follows from  \cite[Theorem 4.5]{KR:NoD}.
\end{proof}

It is worth noting that as compared with Proposition \ref{prop5.3}, in Proposition \ref{prop5.6} we do not assume that (\ref{eq4.12}) is satisfied with $g(|y|)=|y|$.
It follows in particular that  Proposition \ref{prop5.6} covers equations with nonlinearities given in (\ref{eq1.8}).

\subsection{Stability results}

We start with a simple observation. W know that (\ref{eq5.1}) implies (\ref{eq5.2}). Therefore if  $\psi,\psi_n$ satisfy the Hartman--Wintner condition and condition \mbox{\rm(\ref{eq4.4})}, and moreover there is $N\ge1$ such that for every $t\in(0,T]$,
\begin{equation}
\label{eq5.6}
\sup_{n\ge N}\int_{\BR^d}e^{-t\mathrm{Re}\,\psi_n(\xi)}\,d\xi<\infty,
\end{equation}
then (\ref{eq4.14}) is satisfied (the estimate for $p_t$ follows by Fatou's lemma).

\begin{proposition}
Let $\varphi,f$ satisfy the assumptions of Proposition \ref{prop5.3} and 	$\psi,\psi_n$ satisfy the  Hartman--Wintner condition \mbox{\rm(\ref{eq5.1})} and conditions  \mbox{\rm(\ref{eq4.4}), (\ref{eq5.6})}.  Then $u, u_n$  are weak solutions and $u_n\rightarrow u$ weakly in $L^2(Q_T)$.
\end{proposition}
\begin{proof}
By Proposition \ref{prop5.3}, $u$ is a weak solution of (\ref{eq1.1}) and $u_n$ is a weak solution of  (\ref{eq4.3}). Since (\ref{eq5.6}) implies (\ref{eq4.14}), applying Theorem \ref{th4.4} shows that $u_n\rightarrow u$ a.e. in $Q_T$. On the other hand, by (\ref{eq5.4}), $\{u_n\}_{n\ge1}$ is bounded in $L^2(Q_T)$. Therefore it is  weakly relatively compact in $L^2(Q_T)$ which together with the a.e. convergence shows the desired result.
\end{proof}

In case  $L=\Delta^{\alpha/2}$ weak convergence $u_n\rightarrow u$ in $L^2(Q_T)$ was proved in \cite[Theorem 4.2]{XCV} by completely different methods and with additional assumption \cite[(2.7)]{XCV}.

\begin{proposition}
Let $\varphi\in L^1(\BR^d)\cap L^q(\BR^d)$ for some $q\in(1,\infty]$, $f$ satisfy the assumptions of Proposition \ref{prop5.6} and
$\psi,\psi_n$ satisfy the  Hartman--Wintner condition \mbox{\rm(\ref{eq5.1})} and conditions  \mbox{\rm(\ref{eq4.4}), (\ref{eq5.6})}.  Then $u, u_n$  are renormalized  solutions,   $\sup_{n\ge1}\|u_n\|_{L^1(Q_T)}<\infty$, 
$u\in L^1(Q_T)$ and $u_n\rightarrow u$ a.e. in $Q_T$.
\end{proposition}
\begin{proof}
That $u,u_n$ are renormalized solutions follows from Proposition \ref{prop5.6}. Applying Theorem \ref{th4.4} shows that $u_n\rightarrow u$ a.e in $Q_T$. That $\{u_n\}$ is uniformly bounded in $L^1(Q_T)$ and $u\in L^1(Q_T)$ follows from (\ref{eq5.7}).
\end{proof}

\section{Applications}
\label{sec6}

We will give examples of applications of  our general convergence result stated in Theorems \ref{th4.3} and \ref{th4.4}. Clearly, in the case  of L\'evy operators, (\ref{eq4.15}) is satisfied when the transition densities exist. We  shall see that  in many  interesting cases one can easily check (\ref{eq4.4}) and
(\ref{eq4.14}). Therefore, in these cases, one can apply the convergence result to equations with $\varphi$, $f$ satisfying various  assumptions of Theorem \ref{th4.4}, in particular to weak and renormalized solutions considered in Propositions \ref{prop5.3} and \ref{prop5.6}.
In the case of  more general pseudodifferential operators of the form
(\ref{eq3.3}) the situation is quite different. It is known that in some interesting situations there exist transition densities (see, e.g., \cite{KK}), but the problem is whether, for instance, the second part of (\ref{eq4.14}) holds true. For this reason, in the case of  general operators  (\ref{eq3.3}) our results apply to equations with bounded continuous $\varphi$ and bounded $f(\cdot,\cdot,0)$.

In what follows $I_k$ denotes the identity matrix of size $k$.

\subsection{L\'evy operators}

The following sufficient condition for  the Hartman--Wintner conditon and (\ref{eq5.6}) is often easy to check.
Suppose that  there is a characteristic exponent $\tilde \psi:\BR^d\rightarrow\BR$  satisfying  (\ref{eq5.1}) such that for some $N\ge1$ we have
\begin{equation}
\label{eq6.1}
\mathrm{Re}\,\psi_n(\xi)\ge\tilde \psi(\xi),\quad n\ge N,
\end{equation}
for all sufficiently large $|\xi|$. Clearly,  for each $n\ge N$, the function $\psi_n$ satisfies the Hartman--Wintner  condition. Furthermore, since $\tilde\psi$ satisfies (\ref{eq5.1}),  there is $M\ge1$ such that if  $|\xi|\ge M$, then
\[
\exp(-(T-s)\tilde\psi(\xi))
=\exp\Big(-(T-s)\ln|\xi|\frac{\tilde\psi(\xi)}{\ln|\xi|}\Big)
\le\exp(-\kappa\ln|\xi|)
\]
for some constant $\kappa>d$. Therefore from (\ref{eq6.1}) it follows that for a sufficiently large $M$ we have  $\exp(-(T-s)\mathrm{Re}\,\psi_n(\xi))\le|\xi|^{-\kappa}$ if $|\xi|\ge M$ if $n\ge N$. Hence
\[
\int_{|\xi|\ge M}e^{-(T-s)\mathrm{Re}\,\psi_n(\xi)}\,d\xi
\le\int_{\{|\xi|\ge M\}}|\xi|^{-\kappa}\,d\xi =C(M,d)\int^{\infty}_Mr^{-\kappa+d-1}\,dr<\infty,
\]
for $n\ge N$, which implies (\ref{eq5.6}).

Below we provide examples of operators for which the corresponding characteristic exponents  $\psi,\psi_n$ satisfy (\ref{eq4.4}), (\ref{eq5.6}).

\begin{example}
\begin{enumerate}[(i)]
\item  (Fractional Laplace operator). Let $L=\Delta^{\alpha/2}$ for some $\alpha\in(0,2]$. If $\alpha\in(0,2)$, then $L=\II$ with $\nu(dy)=c_{d,\alpha}|y|^{-d-\alpha}$ for some constant $c_{d,\alpha}>0$, and if $\alpha=2$, then $L=\LL=\Delta$. The corresponding characteristic exponent, for $\alpha\in(0,2]$, is $\psi(\xi)=|\xi|^{\alpha}$, $\xi\in\BR^d$.
Let $L^n=\Delta^{\alpha_n/2}$ for some $\alpha_n\in(0,2)$. Then $\psi_n(\xi)=|\xi|^{\alpha_n}$. It is clear that if $\alpha_n\rightarrow \alpha$, then (\ref{eq4.4}), (\ref{eq5.6}) are satisfied.

\item  (Multifractal diffusion). Let
$L=\sum^N_{i=1}c^{(i)}\Delta^{\alpha^{(i)}}$, $L^n=\sum^N_{i=1}c^{(i)}_n\Delta^{\alpha^{(i)}_n}$ with constants $c^{(i)}, c^{(i)}_n\ge0$ and $\alpha^{(i)},\alpha^{(i)}_n\in(0,2]$. The corresponding
characteristic functions are of the form $\psi(\xi)=\sum^N_{i=1}c^{(i)}|\xi|^{\alpha^{(i)}}$, $\psi_n(\xi)=\sum^N_{i=1}c^{(i)}_n|\xi|^{\alpha^{(i)}_n}$, $\xi\in\BR^d$. Assume that $c^{(i)}_n\rightarrow c^{(i)}$, $\alpha^{(i)}_n\rightarrow\alpha^{(i)}$ for $i=1,\dots,N$. Then (\ref{eq4.4}) is satisfied. Furthermore, if $c^{(i)}>0$ for some $1\le i\le N$, then (\ref{eq6.1}) holds with $\tilde\psi(\xi)=
(c^{(i)}/2)|\xi|^{\alpha^{(i)}/2}$.

\item (Operator associated with the relativistic $\alpha$-stable process). Let $L=\psi(D)$ and $L^n=\psi_n(D)$.
\begin{enumerate}
\item $\psi(\xi)=(|\xi|^2+m^{2/\alpha})^{\alpha/2}-m$,
$\psi_n(\xi)=(|\xi|^2+m^{2/\alpha_n})^{\alpha_n/2}-m$, $\xi\in\BR^d$, for some $m>0$ and $\alpha,\alpha_n\in(0,2]$. If $\alpha_n\rightarrow\alpha$, then (\ref{eq4.4}), (\ref{eq6.1}) are satisfied. Note that $L=\Delta$ if $\alpha=2$.

\item $\psi(\xi)=|\xi|^{\alpha}$,
$\psi_n(\xi)=(|\xi|^2+m_n^{2/\alpha})^{\alpha/2}-m_n$, $\xi\in\BR^d$, for some $\alpha\in(0,2)$ and $m_n>0$.
If $m_n\rightarrow0$, then (\ref{eq4.4})  and (\ref{eq6.1}) are satisfied. Note that $L=\Delta^{\alpha/2}$.
\end{enumerate}
\end{enumerate}
\end{example}


\subsection{L\'evy-type operators}

Below we give examples of operators for which (\ref{eq4.6}), (\ref{eq4.4}) are satisfied and consequently  Theorem \ref{th4.3} is applicable.

\begin{example}
\begin{enumerate}[\rm(i)]
\item Let $\sigma=0$, $b=0$, $\gamma(x)=\bar\gamma(x)I_d$ for some bounded Lipschitz continuous  $\bar\gamma:\BR^d\rightarrow\BR$,  and let $N$ be a  $d$-dimensional symmetric $\alpha$-stable process.  Then
    \[
    X^{s,x,i}_t=x_i+\int^t_s\bar\gamma(X^{s,x}_{r-})\,dN^i_r,\quad t\in[s,T],
    \]
and  $\psi(\xi)=|\xi|^{\alpha}$, $\xi\in\BR^d$.  By  (\ref{eq3.20}),
$p(x,\xi)=\psi(\gamma^T(x)\xi)=|\bar\gamma(x)|^{\alpha}|\xi|^{\alpha}$, $x,\xi\in\BR^d$.
The generator with this symbol when restricted to $C^{\infty}_c(\BR^d)$ is of the form
\[
Lu(x)=|\bar\gamma(x)|^{\alpha}\Delta^{\alpha/2}u(x).
\]
Let $L^n$ correspond to the symbol $p_n(x,\xi)=|\bar\gamma^n(x)|^{\alpha_n}|\xi|^{\alpha_n}$ for some $\alpha_n\in(0,2)$ and bounded Lipschitz continuous $\bar\gamma^n:\BR^d\rightarrow\BR$.  Theorem \ref{th4.3}(i) applies in the case where  $\alpha_n\rightarrow\alpha$ and $\bar\gamma^n\rightarrow\bar\gamma$ locally uniformly in $\BR^d$.

\item Let $d=1$, $\sigma=0$, $b=0$ and $\gamma_{1j}(x)=\bar\gamma_j(x)$, $j=1,\dots,k$ for
some bounded Lipschitz continuous  $\bar\gamma_j:\BR\rightarrow\BR$. Assume that
$N^1,\dots,N^k$ are independent  and $N^j$, $j=1,\dots,k$, is a one-dimensional symmetric $\alpha^{(j)}$-stable processes.  We then have
\[
X^{s,x}_t=x+\sum^k_{j=1}\int^t_s\bar\gamma_j(X^{s,x}_{r-})\,dN^j_r,\quad t\in[s,T].
\]
The characteristic exponent of $N$ equals $\psi(\xi)=\sum^k_{j=1}|\xi_j|^{\alpha^{(j)}}$, $\xi\in\BR^k$. By (\ref{eq3.20}),
\[
p(x,\xi)=\psi(\gamma^T(x)\xi)
=\sum^{k}_{j=1}|\bar\gamma_j(x)|^{\alpha^{(j)}}|\xi|^{\alpha^{(j)}},\quad x,\xi\in\BR.
\]
The operator corresponding to the symbol $p$ has the form
\[
Lu(x)=\sum^{k}_{j=1}|\bar\gamma_j(x)|^{\alpha^{(j)}}
(\partial^2_{xx})^{\alpha^{(j)}/2}u(x), \quad u\in C^{\infty}_c(\BR).
\]
Let $L^n$ correspond to the symbol $\sum^{k}_{j=1}|\bar\gamma^n_j(x)|^{\alpha^{(j)}_n}|\xi|^{\alpha^{(j)}_n}$ with some $\alpha^{(j)}_n\in(0,2)$ and bounded Lipschitz continuous $\bar\gamma^n_j:\BR\rightarrow\BR$. If $\alpha^{(j)}_n\rightarrow\alpha^{(j)}$ and $\bar\gamma^n_j\rightarrow\bar\gamma_j$ locally uniformly for $j=1,\dots,k$, then
(\ref{eq4.6}), (\ref{eq4.4}) are satisfied.

\item Let  $d=k=1$, $\sigma,b\in\BR$. Assume that $\gamma:\BR\rightarrow\BR$ is a bounded Lipschitz continuous function and $Z_t=(t,W_t,N_t)$, where $N$ is a one-dimensional symmetric $\alpha$-stable process. Then
\[
p(x,\xi)
=-ib\xi+\frac{a}{2}|\xi|^2+|\gamma(x)|^{\alpha}|\xi|^{\alpha},\quad x,\xi\in\BR,
\]
with $a=\sigma^2$. The corresponding operator has the form
\[
Lu(x)= b\partial_xu(x)+\frac{a}{2}\partial^2_{xx}u(x)
+|\gamma(x)|^{\alpha}(\partial^2_{xx})^{\alpha/2}u(x),\quad u\in C^{\infty}_c(\BR).
\]
Let $L^n$ be the operator with symbol $p_n(x,\xi)=-ib_n\xi+\frac{a_n}{2}|\xi|^2+|\gamma(x)|^{\alpha}|\xi|^{\alpha}$, where $a_n\ge0,b_n\in\BR,\alpha_n\in(0,2)$ and $\gamma_n:\BR\rightarrow\BR$ is
bounded and Lipschitz continuous. If $a_n\rightarrow a$, $b_n\rightarrow b$ and $\alpha_n\rightarrow\alpha$  then clearly (\ref{eq4.6}) and  (\ref{eq4.4}) are satisfied.
\end{enumerate}
\end{example}

\end{document}